\documentclass[11pt]{article}
\usepackage[a4paper]{geometry}
\usepackage{amsthm}
\usepackage{amsmath}
\usepackage{amssymb}
\usepackage{mathtools}
\usepackage{epsfig,graphicx}
\usepackage{hyperref}
\usepackage{subfigure}
\usepackage{xcolor}
\usepackage{bm}

\newtheorem{Theorem}{Theorem}

\newtheorem{Definition}[Theorem]{Definition}

\theoremstyle{remark}
\newtheorem*{Remark}{Remark}

\begin{document}
\title{Simply improved averaging for coupled oscillators and weakly nonlinear waves}

\author{
	Molei Tao
	\thanks{E-mail: \href{mailto:mtao@gatech.edu}{mtao@gatech.edu}}
	\thanks{School of Mathematics, Georgia Institute of Technology, USA} 
}

\maketitle

\abstract{The long time effect of nonlinear perturbation to oscillatory linear systems can be characterized by the averaging method, and we consider first-order averaging for its simplest applicability to high-dimensional problems. Instead of the classical approach, in which one uses the pullback of linear flow to isolate slow variables and then approximate the effective dynamics by averaging, we propose an alternative coordinate transform that better approximates the mean of oscillations. This leads to a simple improvement of the averaged system, which will be shown both theoretically and numerically to provide a more accurate approximation. Three examples are then provided: in the first, a new device for wireless energy transfer modeled by two coupled oscillators was analyzed, and the results provide design guidance and performance quantification for the device; the second is a classical coupled oscillator problem (Fermi-Pasta-Ulam), for which we numerically observed improved accuracy beyond the theoretically justified timescale; the third is a nonlinearly perturbed first-order wave equation, which demonstrates the efficacy of improved averaging in an infinite dimensional setting.
}


\section{Introduction and main results}
\label{sec_intro}
An important class of dynamical problems are highly oscillatory systems, which admit dynamics over at least two timescales with the fastest scale mainly consisting of oscillations. This article considers such systems in which fast oscillations are induced by linear components of differential equations, and additional weak nonlinearities create interactions between these oscillations. More precisely, consider
\begin{equation}
	\dot{X}(t)=\Omega X(t) + \epsilon F(X(t),t),	\qquad X(0)=X_0,
	\label{eq_canonicalSystem}
\end{equation}
where $X(t) \in \mathbb{R}^d \text{ or } \mathbb{C}^d$ (see also Section \ref{sec_PDE} for an infinite dimensional generalization), $\Omega$ is a skew-Hermitian operator (i.e., $\Omega^*=-\Omega$), and $\epsilon \ll 1$. $F$ is a bounded continuous function, with bounded continuous derivative with respect to $X$, and quasiperiodic\footnote{The requirement on quasiperiodicity can be relaxed as long as $\lim_{T\rightarrow\infty} \int_0^T e^{-\Omega t} F\left(e^{\Omega t}Y(t), t \right) dt / T$ exists, but for the simplicity of discussion it will be assumed.} in $t$; i.e., there exist constants $\nu_1,\cdots,\nu_n$ and some function $\hat{F}$ that is 1-periodic in each of its arguments except the first, such that $F(x,t)=\hat{F}(x,\nu_1 t,\cdots,\nu_n t)$ for all $x$ and $t$.

The skew-Hermitianity of $\Omega$ ensures that all its eigenvalues are imaginary. Let $\omega_1,\cdots,\omega_d$ be their absolute values, which represent the frequencies of oscillations originated from the linear term. Note $\omega_1,\cdots,\omega_d$ and $\nu_1,\cdots,\nu_n$ are assumed not to change with $\epsilon$.

At the $\mathcal{O}(1)$ timescale of this system (which will be called the fast or the microscopic scale), the role of the nonlinearity is only perturbative, and it induces an $\mathcal{O}(\epsilon)$ difference from the solution of $\dot{X}=\Omega X$. On the other hand, over the slow/macroscopic timescale of $\mathcal{O}(\epsilon^{-1})$, secular interactions between the linear part and the nonlinearity can globally change the linear solution --- a favorite example of ours is parametric-resonantly perturbed harmonic oscillator, for which one can have arbitrary marcoscopic behavior even though the microscopic behavior is restricted to nearly harmonic oscillations (see Section 3.2 in \cite{TaOw2016} and \cite{XiTa18}).

This article is mainly concerned with the quantification of how the interaction between the fast oscillations accumulates and effectively contributes to the macroscopic dynamics. The celebrated method of averaging already provided a tool for this quantification (see e.g., the monograph of \cite{sanders2007averaging}). To do so, the classical approach is to first introduce a coordinate transform to isolate the slow variable, and then average out the dependence on fast oscillations in its equation of motion. This article, on the other hand, proposes to use a different coordinate transform to improve the accuracy of an averaging approximation.

More precisely, classical averaging introduces $Y(t)=e^{-\Omega t}X(t)$ so that
\[
	\dot{Y}(t)=\epsilon e^{-\Omega t} F\left(e^{\Omega t}Y(t), t \right),	\]
and then this equation is approximated with the right hand side replaced by its time average:
\begin{equation}
	\dot{\bar{Y}}=\epsilon \left\langle e^{-\Omega t} F\left(e^{\Omega t}\bar{Y}, t \right) \right\rangle_t
	\label{eq_YbarRHS}
\end{equation}
See Section \ref{sec_classicalAveraging} for more details.

The improved averaging proposed here uses a different coordinate transform $X \mapsto Z$, implicitly defined via
\[
	X(t)=e^{\Omega t}Z(t)-\epsilon \widehat{\Omega^{-1}} C(Z(t)),
\]
where the additional $\mathcal{O}(\epsilon)$ term is a corrector that aims at accounting for a possibly nonzero mean of the oscillations. Here $\widehat{\Omega^{-1}}$ is essentially $\Omega^{-1}$, but the inverse of its zero eigenvalues (if any) will be replaced by zero. $C(Z)$ is a coarse estimation of the averaged value of the nonlinearity, defined as
\begin{equation}
	C(Z):=\left\langle F\left(e^{\Omega t}Z, t \right) \right\rangle_t .
	\label{eq_CZRHS}
\end{equation}
$Z$'s equation of motion can be written as
\[
	\dot{Z}(t) = \epsilon e^{-\Omega t} \left( F \left( e^{\Omega t}Z-\epsilon P(Z), t \right) - \Omega P(Z) \right) + \mathcal{O}(\epsilon^2),
\]
and we will approximate it by
\begin{equation}
	\dot{\bar{Z}}(t) = \epsilon \left\langle e^{-\Omega t} \left( F \left( e^{\Omega t}\bar{Z}-\epsilon P(\bar{Z}), t \right) - \Omega P(\bar{Z}) \right) \right\rangle_t
	\label{eq_ZbarRHS}
\end{equation}
See Section \ref{sec_improvedMethod} for more details.

Both $\bar{Y}$ and $\bar{Z}$ provide 1st-order approximations of the true solution in the sense that $\bar{Y}(t)-Y(t)=\mathcal{O}(\epsilon)$ and $\bar{Z}(t)-Z(t)=\mathcal{O}(\epsilon)$ till at least $t=\mathcal{O}(\epsilon^{-1})$. However, the proposed method is more accurate because $\dot{Z}-\dot{\bar{Z}}$ is bounded by $\dot{Y}-\dot{\bar{Y}}$ in a sense that will be defined in Section \ref{sec_improvedAccuracy}, where formal theoretical justifications will also be provided. Worth clarifying is, the improved approach is still only 1st-order in $\epsilon$, and generally less accurate than 2nd-order averaging; however, the algebra of 2nd-order averaging can be formidable for high-dimensional problems.

In addition to the theoretical discussion, the improved accuracy will also be numerically quantified and demonstrated by three examples, the first two being oscillators coupled through weak nonlinearities, and the third being a nonlinearly perturbed first order wave equation, which is infinite dimensional but still within the scope of \eqref{eq_canonicalSystem}. More specifically, Section \ref{sec_CPUT} analyzes a new engineering device for wireless energy transfer, where the improved averaging allows a performance quantification for which classical averaging is not accurate enough, and the obtained analytical results guided us to design device parameters for its desired operations in \cite{CPUT2018IEEE}. Section \ref{sec_FPU} records improved numerical accuracies on a classical test problem of Fermi-Pasta-Ulam, and the improvement was beyond the theoretically justified timescale but for a longer time. Section \ref{sec_PDE} compares how classical averaging and improved averaging capture the long time behavior of weakly nonlinear waves modeled by an advection-reaction PDE; for this problem, even though our method generalizes to infinite dimensions, the numerical computations are still conducted after spatial discretization, in this case via pseudospectral method.

Time averages such as the right hand sides of \eqref{eq_YbarRHS}, \eqref{eq_CZRHS}, \eqref{eq_ZbarRHS} play an essential role in averaging approximations, and they can be computed either analytically or numerically. Section \ref{sec_CPUT} conducts this analytically, Section \ref{sec_FPU} does it both analytically and numerically because the analytical expressions are too complex, and Section \ref{sec_PDE} mainly does it numerically because an analytical result may no longer be available. To compute time averages numerically, we use the standard composite trapezoidal rule when the function to be averaged is periodic with known period; in other cases (when the function is quasiperiodic or with unknown period), we use the recently developed tool of weighted Birkhoff averaging \cite{das2017quantitative}, which achieves accuracy at a relatively low computational cost. Details of the time-average computations are described in Section \ref{sec_numAvg}.

While the focus of this article is the alternative coordinate transformation that improves averaging, one of its implications is an efficient numerical simulation of highly oscillatory problems, based on integrating the averaged equations. For the sake of length, we will discuss the details of this application in a separate paper. Nevertheless, it is necessary to mention great existing works in the active field of highly oscillatory system simulations, and the list by no means can be complete. For instance, (i) since the formulation of envelope following method \cite{petzold1997numerical}, clever approaches that address one high-frequency (which, in the case of eq.\ref{eq_canonicalSystem}, corresponds to $\Omega$ with eigenvalues all being integer multiples of one value) have been continuously constructed, such as multi-revolution composition method \cite{chartier2014multi,chartier2017convergence}, stroboscopic averaging method (which can be made high-order in $\epsilon$) \cite{calvo2012stroboscopic,calvo2011numerical,castella2015stroboscopic}, the two-scaled reformulation (which is uniformly accurate after adding one artificial dimension) \cite{chartier2015uniformly}, and additional uniformly accurate approaches based on various formal asymptotic expansions \cite{bao2012uniformly,bao2014uniformly,bao2016uniformly, crouseilles2017nonlinear}. Tools for performance analysis have been proposed too, such as modulated Fourier expansion \cite{hairer2000long, cohen2008long}, which in turn accelerates the development of numerical methods, for instance, for second-order differential equations (see e.g., a seminal work \cite{cohen2003modulated}, a specific investigation \cite{sanz2009modulated}, and a recent work \cite{zhao2017uniformly} in which 2nd-order uniform accuracy was achieved on half of the variables for polynomial potentials). More contributions to second-order equations include \cite{LeBris:07, CaCha:09}.
(ii) When there are multiple high-frequencies, exponential and trigonometric integrators can sometimes provide accurate simulations, and we refer to \cite{hochbruck2010exponential, hochbruck1998exponential, cox2002exponential, kassam2005fourth, Grubmuller:91, Tuckerman:92, wisdom1991symplectic, gautschi1961numerical, deuflhard1979study, Skeel:99, hochbruck1999gautschi, hairer2000long, moler2003nineteen,saad1992analysis,hochbruck1997krylov,SIM2} (see Section \ref{sec_numAvg} for more discussions). An interesting alternative idea is to approximate multiple frequencies by integer multiples of one frequency, which was realized by combining the two-scale reformulation and multi-revolution composition \cite{chartier2017highly}. Another direction of developments was based on the numerical integration of highly oscillatory functions \cite{iserles2004numerical, iserles2004quadrature, iserles2005efficient, levin1996fast}, and mainly investigated were the simulations of second-order equations \cite{khanamiryan2008quadrature,wang2013filon,wang2016arbitrary}; usually one will obtain an iterative method due to implicitness. 
Related is an application of Hamiltonian Boundary Value Methods, which constructs implicit methods for Newtonian second-order problems \cite{brugnano2017effectiveness}. Also worth mentioning is \cite{condon2010second}, which considered a 1 degree-of-freedom, however possibly nonlinear system, subject to highly oscillatory forcing terms. (iii) Another important direction is to construct methods that utilize the averaging operator $\langle \cdot \rangle$ more directly. One notable contribution to linear oscillations coupled via nonlinearity is \cite{haut2014asymptotic}, which was based on classical averaging \eqref{eq_YbarRHS} (with additional help from Parareal iterations). We also mention \cite{LeRe01} which averages forces in Newtonian second-order problems. 
Remarkably, continuous efforts have been made to develop up-scaled solvers for systems with more general fast dynamics not limited to multi-frequency oscillatory. These include the Heterogeneous Multiscale Methods (see \cite{HMMoriginal, MR2314852, abdulle2012heterogeneous} for reviews and \cite{Ariel:08, calvo2010heterogeneous, ariel2013multiscale, ariel2016parareal} for specific investigations on highly oscillatory problems), equation-free approaches (e.g., \cite{MR2041455, kevrekidis2004equation, KevGio09}), discussions on Young measure \cite{Art07b, Art07}, FLAVORs \cite{FLAVOR10,FLAVORPDE11} (which are inspired by the formers but adapted to problems with preserved structures such as symplecticity or invariant distribution), and more recently, some promising ideas that aimed at iteratively improving the accuracy of general multiscale numerical methods via the Parareal parallelism \cite{ariel2016parareal}.

\section{Methods}
\label{sec_method}
\subsection{A review of the classical averaging}
\label{sec_classicalAveraging}
Having both $\mathcal{O}(1)$ and $\mathcal{O}(\epsilon)$ terms on the right hand side of \eqref{eq_canonicalSystem} makes the dynamics to contain a mixture of both slow and fast components. 
To quantify the long term effect of the $\mathcal{O}(\epsilon)$ term (i.e., the nonlinearity), a classical approach (see e.g., \cite{sanders2007averaging}) first introduces the coordinate transform
\begin{equation}
	Y(t)=e^{-\Omega t}X(t),	\qquad \text{which leads to}
	\label{eq_transformOld}
\end{equation}
\begin{equation}
	\dot{Y}(t)=\epsilon e^{-\Omega t} F\left(e^{\Omega t}Y(t), t \right),	\qquad Y(0)=X_0.
	\label{eq_oldUnaveraged}
\end{equation}
This way, a slow variable $Y$ is isolated (note its rate of change is $\dot{Y}=\mathcal{O}(\epsilon)$). However, $Y$ dynamics are still modulated by fast oscillations due to the $t$ dependence in the RHS of \eqref{eq_oldUnaveraged}. Therefore, the method of averaging is then employed to remove the fast oscillations to allow a focus on the purely slow dynamics: let
\begin{equation}
	\mathcal{F}(Y,t):= e^{-\Omega t} F\left(e^{\Omega t}Y, t \right)
	\quad \text{and} \quad
	\bar{\mathcal{F}}(Y) := \left\langle \mathcal{F}(Y,t) \right\rangle_t 
	\label{eq_oldAveragedVectorField}
\end{equation}
where the time averaging operator is generally defined as
\begin{equation}
	\langle \mathcal{F}(x,t) \rangle_t := \lim_{T\rightarrow\infty} \frac{1}{T}\int_0^T \mathcal{F}(x,t) dt
	\label{eq_curlyF}
\end{equation}
(note in the special case of $\mathcal{F}$ being $T$-periodic in $t$, we also have $\langle \mathcal{F}(x,t) \rangle_t =  \frac{1}{T}\int_0^T \mathcal{F}(x,t) dt$), and consider the averaged dynamics
\begin{equation}
	\dot{\bar{Y}}=\epsilon \bar{\mathcal{F}}(\bar{Y}),	\qquad \bar{Y}(0)=X_0,
	\label{eq_OldAveraged}
\end{equation}
then $\|\bar{Y}(t)-Y(t)\|=\mathcal{O}(\epsilon)$ till at least $t=\mathcal{O}(\epsilon^{-1})$ (see \cite{sanders2007averaging} for details; note the quasiperiodicity of $F(X,t)$ and $\omega_i$'s and $\nu_j$'s being independent of $\epsilon$ are needed for ensuring this order of accuracy). 

The advantage of \eqref{eq_OldAveraged} is, as $\bar{Y}$ changes at the $\mathcal{O}(\epsilon^{-1})$ timescale (which can be seen via a time change in \eqref{eq_OldAveraged}), it is now purely slow. Insights about the system's effective dynamics can thus be gained by analyzing \eqref{eq_OldAveraged}. In fact, once $\bar{Y}(t)$ is known, $\hat{X}(t):=e^{\Omega t}\bar{Y}(t)$ will approximate $X(t)$, with error also $\mathcal{O}(\epsilon)$ till at least $t=\mathcal{O}(\epsilon^{-1})$.

\subsection{The simple improvement}
\label{sec_improvedMethod}
The approximate solution $\hat{X}$ contains a potential source of inaccuracy that can be corrected. To see the origin of the idea, consider a simplest ``nonlinearity'' of $F(X(t),t)=C$, which is a constant vector. The exact solution to \eqref{eq_canonicalSystem} is then
\[
	X(t)=e^{\Omega t} \big( X_0+\epsilon \Omega^{-1} C \big) - \epsilon \Omega^{-1} C,
\]
and it has a small but nonzero time average of $- \epsilon \Omega^{-1} C$, which is however not reflected in the approximation $\hat{X}(t)=e^{\Omega t}\bar{Y}(t)=e^{\Omega t}X_0$ obtained from classical averaging \eqref{eq_OldAveraged}. It is easy to see, however, that if one uses a coordinate transform of $Z=\exp(-\Omega t)(X+\epsilon\Omega^{-1}C)$ instead of $Y=\exp(-\Omega t)X$, this nonzero average will be recovered, and the exact solution too.

To extrapolate this idea to general problems, let
\begin{equation}
	C(Z):=\left\langle F\left(e^{\Omega t}Z, t \right) \right\rangle_t := \lim_{T\rightarrow\infty} \frac{1}{T}\int_0^T F\left(e^{\Omega t}Z, t \right) \, dt
	\label{eq_CZ}
\end{equation}
and define a coordinate transformation from $X$ to $Z$ implicitly via
\begin{equation}
	X(t)=e^{\Omega t}Z(t)-\epsilon P(Z(t)),
	\label{eq_transformNew}
\end{equation}
where $P(Z):=\Omega^{-1} C(Z)$ if $\Omega$ is invertible; otherwise we will assume $\Omega=V^{-1}[\omega_i]_{ii} V$, define $\widehat{\Omega^{-1}}:=V^{-1}[\widehat{\omega_i^{-1}}]_{ii} V$ where $\widehat{\omega_i^{-1}}$ is $\omega_i^{-1}$ if $\omega_i\neq 0$, and 0 otherwise, and let $P(Z):=\widehat{\Omega^{-1}} C(Z)$.

Taking the time derivative of \eqref{eq_transformNew} yields
\[
	\left( e^{\Omega t}-\epsilon P'(Z) \right) \dot{Z}(t) = \epsilon \left( F \left( e^{\Omega t}Z-\epsilon P(Z), t \right) - \Omega P(Z) \right),
\]
and therefore the equation of motion in the new coordinate can be written as
\begin{equation}
	\dot{Z}(t) = \epsilon e^{-\Omega t} \left( F \left( e^{\Omega t}Z-\epsilon P(Z), t \right) - \Omega P(Z) \right) + \mathcal{O}(\epsilon^2).
	\label{eq_newUnaveraged}
\end{equation}
The right hand side is $\mathcal{O}(\epsilon)$ and averaging can again provide an approximation: let 
\begin{equation}
	G(Z,t) := e^{-\Omega t} \left( F \left( e^{\Omega t}Z-\epsilon P(Z), t \right) - \Omega P(Z) \right), \qquad 
	\bar{G}(Z) := \left\langle G(Z,t) \right\rangle_t,
	\label{eq_newAveragedVectorField}
\end{equation}
and consider
\begin{equation}
	\dot{\bar{Z}} = \epsilon \bar{G}(\bar{Z}),
	\label{eq_newAveraged}
\end{equation}
then $\bar{X}(t):=e^{\Omega t}\bar{Z}(t)-\epsilon P(\bar{Z}(t))$ will be an $\mathcal{O}(\epsilon)$ approximation of $X(t)$ till at least $t=\mathcal{O}(\epsilon^{-1})$.

\begin{Remark}
	\eqref{eq_newAveragedVectorField} is equivalent to $\bar{G}(Z) = \left\langle e^{-\Omega t} F \left( e^{\Omega t}Z-\epsilon P(Z), t \right) \right\rangle_t$ because $\Omega P(Z)$ is independent of time and thus averaged out after being multiplied by $e^{-\Omega t}$. However, we prefer to use \eqref{eq_newAveragedVectorField} because although the exact time average of $e^{-\Omega t} \Omega P(Z)$ is zero, in Sections \ref{sec_numAvg}, \ref{sec_FPU} and \ref{sec_PDE} the time average will be approximated numerically, and keeping this term led to smaller errors in those experiments.
\end{Remark}

\begin{Remark}
	The initial condition of $\bar{Z}(0)$, in general, is no longer the same as $X(0)=X_0$ due to the $P(\cdot)$ term in \eqref{eq_transformNew}. When the nonlinearity of $P$ may prevent $\bar{Z}(0)$ from being exactly solvable as an explicit function of $X_0$, one can expand $\bar{Z}(0)$ as an asymptotic series in $\epsilon$ and match the leading orders to obtain the following approximation:
	\[
		\bar{Z}(0)=X_0+\epsilon P(X_0) + \mathcal{O}(\epsilon^2)
	\]
	Given the remainder is $\mathcal{O}(\epsilon^2)$, using $\bar{Z}(0) \approx X_0+\epsilon P(X_0)$ in the proposed averaged system \eqref{eq_newAveraged} will only induce an $\mathcal{O}(\epsilon^2)$ error till at least $t=\mathcal{O}(\epsilon^{-1})$, and thus the order of accuracy in $\bar{Z}$ or $\bar{X}$ is not affected.
\end{Remark}

\subsection{The improved accuracy}
\label{sec_improvedAccuracy}
It is intuitive to think that \eqref{eq_newAveraged} can provide a more accurate approximation than \eqref{eq_OldAveraged} because the former better characterizes the mean value of the oscillations, which is not necessarily zero. This intuition can be made precise in the following way.

\begin{Definition}
	Given column-vector-valued functions $w_1(t,x)$ and $w_2(t,x)$, consider an $x$-indexed inner product defined by
	\[
		(w_1 \cdot w_2)(x) := \langle w_1(t,x)^T w_2(t,x) \rangle_t.
	\]
	Denote by $\|\cdot\|(x)$ the induced norm.
\end{Definition}
Since \eqref{eq_newUnaveraged} can be alternatively written as
\[
	\dot{Z}(t) = \epsilon e^{-\Omega t} \left( F \left( e^{\Omega t}Z, t \right) - \Omega P(Z) \right) + \mathcal{O}(\epsilon^2),
\]
let
\begin{equation}
	\mathcal{G}(Z,t) := e^{-\Omega t} \left( F \left( e^{\Omega t}Z, t \right) - \Omega P(Z) \right), \quad\text{and}\quad 
	\bar{\mathcal{G}}(Z) := \left\langle G(Z,t) \right\rangle_t,
	\label{eq_newAveragedVectorFieldAlternative}
\end{equation}
and then the averaged vector field $\bar{\mathcal{G}}$ will also provide an approximation at the same order of accuracy as \eqref{eq_newAveraged} due to 1st-order averaging theory \cite{sanders2007averaging}. 
\begin{Theorem}
	Consider $\mathcal{F}$, $\bar{\mathcal{F}}$ in \eqref{eq_oldAveragedVectorField} and $\mathcal{G}$, $\bar{\mathcal{G}}$ in \eqref{eq_newAveragedVectorFieldAlternative}. Let $\Delta\mathcal{F}(x,t)=\mathcal{F}(x,t)-\bar{\mathcal{F}}(x)$ and $\Delta\mathcal{G}(x,t)=\mathcal{G}(x,t)-\bar{\mathcal{G}}(x)$. Then, $\big( (\Delta\mathcal{F}-\Delta\mathcal{G})\cdot \Delta\mathcal{G} \big)(x) = 0$ for any $x$ and therefore
	\[
		\|\Delta\mathcal{G}\|(x) \leq \|\Delta\mathcal{F}\|(x).
	\]
	\label{thm_compareAvgVectorFields}
\end{Theorem}
\begin{proof}
	Note $\left\langle e^{-\Omega t}\Omega P(x) \right\rangle_t = 0$ and thus $\bar{\mathcal{F}}(x)=\bar{\mathcal{G}}(x)$. Recall $\Omega P(x)=C(x)=\left\langle F\left(e^{\Omega t} x, t \right) \right\rangle_t$, and therefore,
	\begin{align*}
		&\quad \big( (\Delta\mathcal{F}-\Delta\mathcal{G})\cdot \Delta\mathcal{G} \big)(x) = \big( (\mathcal{F}-\mathcal{G})\cdot (\mathcal{G}-\bar{\mathcal{G}}) \big)(x) \\
		&= \left\langle 
			\left( e^{-\Omega t}\Omega P(x) \right)^T \left( e^{-\Omega t} \left( F\left(e^{\Omega t} x, t \right) - \Omega P(x) \right) - \bar{\mathcal{G}}(x) \right)
		   \right\rangle_t \\
		&= \left\langle C(x)^T \left( F\left(e^{\Omega t} x, t \right) - C(x) \right) \right\rangle_t - \left\langle \left(e^{-\Omega t}C(x) \right)^T \bar{\mathcal{G}}(x) \right\rangle_t \\
		&= C(x)^T \left\langle F\left(e^{\Omega t} x, t \right) \right\rangle_t - C(x)^T C(x) - C(x)^T \left\langle e^{\Omega t} \right\rangle_t \bar{\mathcal{G}}(x) \\
		&= C(x)^T C(x) - C(x)^T C(x) - 0 = 0
	\end{align*}
	This orthogonality naturally leads to $\|\Delta\mathcal{F}\|^2=\|\Delta\mathcal{G}\|^2+\|\Delta\mathcal{F}-\Delta\mathcal{G}\|^2$ and thus $\|\Delta\mathcal{F}\| \geq \|\Delta\mathcal{G}\|$.
\end{proof}

Theorem \ref{thm_compareAvgVectorFields} shows that the proposed approach has smaller difference between the unaveraged and averaged vector field than the classical approach. The following formal argument shows that this reduced difference decreases the amount of fluctuations in the true solution not captured by averaging: consider
\begin{equation}
	\dot{Y} = \epsilon \mathcal{F}(Y,t) + \mathcal{O}(\epsilon^2), \qquad
	\dot{Z} = \epsilon \mathcal{G}(Z,t) + \mathcal{O}(\epsilon^2).
	\label{eq_beforeNearId}
\end{equation}
Assume $\epsilon$ is small enough such that there exist near-identity transformations
\begin{equation}
	Y(t) = \bar{Y}(t)+\epsilon u(\bar{Y}(t),t) + \mathcal{O}(\epsilon^2) \qquad
	Z(t) = \bar{Z}(t)+\epsilon v(\bar{Z}(t),t) + \mathcal{O}(\epsilon^2)
	\label{eq_nearId}
\end{equation}
that implicitly define $\bar{Y}$ and $\bar{Z}$ such that
\begin{equation}
	\dot{\bar{Y}} = \epsilon \bar{\mathcal{F}}(\bar{Y}) + \mathcal{O}(\epsilon^2) \qquad
	\dot{\bar{Z}} = \epsilon \bar{\mathcal{G}}(\bar{Z}) + \mathcal{O}(\epsilon^2).
	\label{eq_afterNearId}
\end{equation}
Then the transformations must satisfy certain conditions. To find these conditions, substituting \eqref{eq_nearId} into \eqref{eq_beforeNearId} and using \eqref{eq_afterNearId} to match $\mathcal{O}(\epsilon)$ terms yields the first-order homological equations
\begin{equation}
	\bar{\mathcal{F}}(y)+\partial_t u(y,t) = \mathcal{F}(y,t),	\qquad
	\bar{\mathcal{G}}(z)+\partial_t v(z,t) = \mathcal{G}(z,t),
	\label{eq_homological}
\end{equation}
where $y$ and $z$ are viewed as parameters. The choices of
\begin{equation}
	\bar{\mathcal{F}}(y) = \langle \mathcal{F}(y,t) \rangle_t,	\qquad
	\bar{\mathcal{G}}(z) = \langle \mathcal{G}(z,t) \rangle_t
	\label{eq_solvability}
\end{equation}
guarantee the existence of $u$ and $v$ and reveal that they correspond to leading order fluctuations. To understand such choices, think about the easiest case when $\mathcal{F}(\cdot,t)$ is periodic in $t$, and then \eqref{eq_solvability} is the solvability condition for \eqref{eq_homological} to admit periodic-in-$t$ solutions.
	
	Theorem \ref{thm_compareAvgVectorFields} shows that $v$ is smaller than $u$ in the sense that $\|\partial_t v\|(x) \leq \|\partial_t u\|(x)$ for all $x$. Since $v$ and $u$ are respectively the leading-order fluctuations of the exact solution (in different coordinates) around $\bar{Z}$ and $\bar{Y}$, $\bar{Z}$ approximates the exact solution more accurately than $\bar{Y}$.

\begin{Remark}
	In practice, our preference is to use the averaged vector field $\epsilon\bar{G}$ (eq.'s \ref{eq_newAveragedVectorField} and \ref{eq_newAveraged}) instead of $\epsilon\bar{\mathcal{G}}$. Their difference is $\mathcal{O}(\epsilon^2)$, and the above discussion, which is based on 1st-order averaging / normal form theory, cannot tell this difference apart and remain inconclusive for comparing $\epsilon\bar{G}$ and $\epsilon\bar{\mathcal{G}}$. However, $\bar{G}$ was obtained with less approximations, and in numerical experiments considered in this article it produced better accuracy.
\end{Remark}

\begin{Remark}
	Although the proposed approach has improved accuracy, it is based on 1st-order averaging and its accuracy is still only $\mathcal{O}(\epsilon)$. It is possible to perform higher-order averaging to increase the order of accuracy; however, significantly more terms will need to be computed and the $X$ derivative(s) of $F$ will be involved, which can be a significant computational overhead especially when $X$ is high-dimensional. In addition, the result may be difficult to interpret because higher-order-averaged systems usually have much more involved vector fields, which can complicate the analysis of their dynamics.
\end{Remark}

\subsection{A brief discussion of numerical averaging}
\label{sec_numAvg}
Both classical averaging and the proposed improvement require the computation of $\langle f(x,t) \rangle_t$ for (quasi)periodic $f(x,\cdot)$. It is not always the case that this time average can or should be analytically computed. For instance,
\begin{itemize} 
\item
For the Fermi-Pasta-Ulam example discussed in Section \ref{sec_FPU}, $\bar{\mathcal{F}}$ (see \eqref{eq_oldAveragedVectorField}) has a lengthy expression (see the Appendix), which not only makes the analysis of the averaged system \eqref{eq_OldAveraged} difficult, but also slows down the numerical simulation due to expensive function evaluations; unfortunately, the expression of $\bar{G}$ (see \eqref{eq_newAveragedVectorField}) is only much lengthier (details available upon request).

\item
For the advection-reaction PDE example discussed in Section \ref{sec_PDE}, $\bar{\mathcal{F}}$ may or may not be analytically computable; meanwhile, $\bar{G}$ or $C(Z)$ \eqref{eq_CZ} do not admit closed form expressions due to the involvement of non-local spatial operations.

\item
Even finite-dimensional problem can require averages that are not exactly computable; an example is
\[
	\frac{1}{2\pi}\int_0^{2\pi} \frac{1}{\sqrt{(a+\cos t)^2+(b+\sin t)^2}} dt,
\]
which naturally originates from celestial mechanics or molecular dynamics when one tries to average gravitational / electric potential over a periodic orbit. The lack of closed form expression for this integral was one of the motivations for the introduction of elliptic functions (see e.g., \cite{abramowitz1964handbook}).
\end{itemize}

\noindent
Therefore, we numerically approximate time averages when necessary. Given $f(x,\cdot)$ with fixed $x$ (note $x$'s value changes when an averaged system is numerically simulated; also, in the improved method (Section \ref{sec_improvedMethod}) one needs to average two different functions, respectively for computing $C(\cdot)$ and $\bar{G}(\cdot)$), this is done in the following way:
\begin{itemize}
\item
	When $f(x,\cdot)$ is known to be $T$-periodic, composite trapezoidal rule is employed, i.e., choose a large enough $N$ and use the approximation
	\begin{equation}
		\langle f(x,\cdot) \rangle_t \approx \frac{1}{N} \sum_{i=0}^{N-1} f(x,t_0+iT/N).
		\label{eq_numAveragingPeriodic}
	\end{equation}
	In theory, $t_0$ can be any value as $N\rightarrow \infty$. In practice, when $\langle f(x,\cdot) \rangle_t$ is used in the numerical simulation of an averaged system, we simply choose $t_0$ to be the time corresponding to the current timestep/stage. This choice is because, in the special case of $N=1$, the simulation of the transformed system \eqref{eq_oldUnaveraged} will degenerate to a subset of the popular exponential integrators for the untransformed system \eqref{eq_canonicalSystem}. (See \cite{hochbruck2010exponential} for a review of exponential integrators, \cite{hochbruck1998exponential, cox2002exponential, kassam2005fourth} for an additional but far-from-complete list of contributions with additional references within, \cite{Grubmuller:91, Tuckerman:92, wisdom1991symplectic, gautschi1961numerical, deuflhard1979study, Skeel:99, hochbruck1999gautschi, hairer2000long} for closely related popular methods, and e.g., \cite{moler2003nineteen,saad1992analysis,hochbruck1997krylov} for the computations of matrix exponential and \cite{SIM2} for exponentiating a sequence of matrices.)
	
	\quad Given periodicity, \eqref{eq_numAveragingPeriodic} is a very good approximation as long as $f$ is also smooth in time, because under this condition composite trapezoidal rule converges faster than any power of $N$ (e.g., \cite{atkinson2008introduction}).
\item
	When $f(x,\cdot)$ is only quasiperiodic, the method of weighted Birkhoff averaging \cite{das2017quantitative} is employed instead, i.e., fix an arbitrary $\delta t$ that can resolve all frequencies, choose a large enough $N$ independent of $\delta t$, and use the approximation
	\begin{equation}
		\langle f(x,\cdot) \rangle_t \approx \frac{\sum_{i=0}^{N-1} w_i f(x,t_0+i\delta t)}{\sum_{i=0}^{N-1} w_i}
		\label{eq_numAveragingQuasiperiodic}
	\end{equation}
	\begin{equation}
		\text{where}\quad w_i=\exp\left(-\frac{1}{t(1-t)}\right) \quad\text{with}\quad t=\frac{i+1}{N+1}.
		\label{eq_weightsBirkhoff}
	\end{equation}
	\eqref{eq_numAveragingQuasiperiodic} is a good approximation because it was shown in \cite{das2017quantitative} that, provided smoothness of $f$, the convergence speed is also faster than any power of $N$, unless the rotation vector of the quasiperiodicity (or equivalently, the frequencies $\omega_1,\cdots,\omega_d$ and $\nu_1,\cdots,\nu_n$) is from a measure-zero set\footnote{When the rotation vector is not of Diophantine class $\beta$ for any $\beta>0$, the convergence speed slows down to polynomial in $N$; see \cite{das2017quantitative} for more details.}.
	
	\quad Note that no knowledge about the rotation number is needed, although its value affects the accuracy of \eqref{eq_numAveragingQuasiperiodic} when a finite value of $N$ is chosen. Great methods for estimating the rotation number from a finite time trajectory exist, such as \cite{laskar1993frequency1, luque2014quasi}, but such an estimation is not necessary here.
\item
	When $f(x,\cdot)$ is periodic with unknown period, we still use the weighted Birkhoff average \eqref{eq_numAveragingQuasiperiodic} as an approximation.

\item
	How to choose a large enough $N$, i.e., the number of samples, to well approximate the time-average? An automated two-step procedure can be used: (i) at the initialization stage, numerically average $f(x,t)=\exp(\Omega t)$ using $N$ samples, monitor its convergence by checking how much \eqref{eq_numAveragingQuasiperiodic} deviates from the true average, which is known to be zero, and then tune $N$ accordingly until a desired accuracy is accomplished; (ii) during the numerical simulation of the averaged system, when a different $f(x,t)$ is being time-averaged (e.g., $\mathcal{F}$ in \eqref{eq_curlyF}), start with the $N$ value found in step (i), and keep on increasing it until the change in the averaged value is small enough. A remark is that for the special case where $F(X,t)$ in \eqref{eq_canonicalSystem} is independent of $t$, we found step (ii) unnecessary in our experiments.
\end{itemize}
Important to remark is, the powerful and versatile idea of introducing weights for improved numerical averaging is not new in this article. Its variations have been used in multiple contexts, serving as an important component of various successful numerical methods, including Heterogeneous Multiscale Methods (e.g., \cite{HMMoriginal,MR2165382,Eric07HMMlike,MR2314852,Ariel:08,Seamless09,calvo2010heterogeneous,abdulle2012heterogeneous,brumm2013heterogeneous,ariel2013multiscale}), mollified impulse methods (e.g., \cite{Skeel:99, izaguirre1999longer, Sanz-Serna:08}), and frequency map analysis (e.g., \cite{laskar1993frequency1, laskar1993frequency2, laskar1999introduction}). In fact, numerically constructing the averaged system \eqref{eq_YbarRHS} or \eqref{eq_ZbarRHS} and then simulating it with a larger timestep is a method that can be viewed as a variation of Heterogeneous Multiscale Methods, with additional help from FLow AVeraging integratORs \cite{FLAVOR10,FLAVORPDE11}, because this method first flows vector fields and then averages them.

How computationally-efficient will be the simulation of numerically-averaged system? The averaging per se is costly as it is performed at each time step. However, to simulate the system till $\mathcal{O}(\epsilon^{-1})$ time, one can use an $o(\epsilon^{-1})$-sized timestep because the averaged system is at the $\mathcal{O}(\epsilon^{-1})$ timescale. This way, the overall computational cost scales linearly with $N$ but remains independent of $\epsilon$. Therefore, the method is more computationally advantageous when $\epsilon$ becomes smaller. Sections \ref{sec_FPU} and \ref{sec_PDE} provide quantifications of such an advantage, and more details on numerical averaging will be discussed in a separate paper.

\section{Demonstration 1: Capacitive Parametric Ultrasonic Transducers}
\label{sec_CPUT}

\subsection{The model}
Recently a new engineering device for wireless energy transfer called Capacitive Parametric Ultrasound Transducer (CPUT) has been designed. It is based on parametric resonance in an RLC circuit, where the capacitance of the circuit is mechanically modulated by ultrasound so that mechanical energy was injected into the circuit and converted to electrical energy \cite{surappa2017capacitive, CPUT2018IEEE}. One of its goals is to utilize the low attenuation of ultrasound in vivo to safely and efficiently deliver power to implanted medical devices. Such a system was modeled in \cite{CPUT2018IEEE} as two oscillators coupled via nonlinearity and perturbed in time, and this model will be analyzed here to provide design guidance and performance quantification.

The model corresponds to the system
\begin{equation} \begin{dcases}
	\frac{d^2 V}{ds^2} + \frac{R}{L} \frac{dV}{ds} + \frac{d-x}{LA\varepsilon_0} V = 0 \\
	m \frac{d^2 x}{ds^2} + b \frac{dx}{ds} + k x = F_0 \sin(2\omega_0 s) + \frac{\varepsilon_0 A}{2} \frac{V^2}{(d-x)^2},
\end{dcases} 
\label{eq_CPUTwithUnits}
\end{equation}
where $V(s)$ and $x(s)$ respectively correspond to the voltage across the capacitor and the displacement of the parallel plates of the capacitor at time $s$. The external ultrasound exerts a periodic force on the membrane, which constitutes as a plate of the capacitor.

The intuition is, when the ultrasound frequency of $2\omega_0$ matches the intrinsic frequency of the mechanical component (i.e., the $x$ equation in \eqref{eq_CPUTwithUnits} without the nonlinear term on the right hand side) and twice the intrinsic frequency of the electric component (i.e. $V$ again without nonlinearity), the term $xV$ makes the first equation resemble Mathieu's equation \cite{MathieuEquation}, which exhibits parametric resonance (see e.g., the review of \cite{verhulst2009perturbation}) and thus increased energy when $R/L$ is small enough. However, the true dynamics of \eqref{eq_CPUTwithUnits} is more complex than this intuition; for example, physically speaking $x$ should not exceed $d$ as the plate displacement cannot be greater than the initial double-plate distance, and this is mathematically reflected in the fact that the nonlinearity will prevent an unbounded growth of the solution. Some questions are, does the energy growth saturate, and how much output can one get? A quantitative analysis of the effect of the nonlinearity is therefore desired.

To do so, we first identify a small parameter by normalizing numerical values of the parameters: introduce scaling factors $c$ and $\mu$ (for design parameters used in \cite{CPUT2018IEEE}, $c=10^7$ and $\mu=10^8$), rescale time by letting $t=c s$, and rescale the plate displacement by letting $y=\mu x$. Simplify notations by letting:
\[
\epsilon=\frac{\mu^3 \varepsilon_0 A}{2m c^2}, ~
\gamma=\frac{R}{Lc\epsilon}, ~
\alpha=\frac{1}{L A \varepsilon_0 c^2 \mu \epsilon}, ~
\beta=\frac{b}{m c \epsilon}, ~
F=\frac{\mu F_0}{m c^2 \epsilon}, ~
D=\mu d, ~
\omega=\frac{\omega_0}{c}.
\]
Then, under the assumption that parameters sufficiently satisfy matching conditions, namely $\frac{d}{LA\varepsilon_0 c^2}-\frac{\omega_0^2}{c^2}=\mathcal{O}(\epsilon^2)$ and $\frac{k}{m c^2}-\frac{4\omega_0^2}{c^2}=\mathcal{O}(\epsilon^2)$, the governing equations rewrite as:
\begin{equation} \begin{dcases}
	\ddot{V} + \epsilon \gamma \dot{V} + \big(\omega^2+\mathcal{O}(\epsilon^2)\big)V = \epsilon \alpha y V , \\
	\ddot{y} + \epsilon \beta \dot{y} + \big((2\omega)^2+\mathcal{O}(\epsilon^2)\big)y = \epsilon F \sin(2\omega t) + \epsilon \frac{V^2}{(D-y)^2} .
\end{dcases} 
\label{eq_CPUToriginal}
\end{equation}
This way, averaging can be used to approximate \eqref{eq_CPUToriginal} for a better understanding of its behavior (see Section \ref{sec_classicalAveraging}). We refer to \cite{nayfeh2008perturbation, nayfeh2011method, verhulst2009perturbation, DoVe08, verhulst2002parametric, fatimah2002bifurcations, verhulst2005autoparametric, zounes2002subharmonic, KoOw13} for an additional but incomplete list of similar analyses / applications. Besides averaging, many other powerful methods exist, such as the method of multiple scales (e.g., \cite{kevorkian2012multiple, rand1980bifurcation, rand2012lecture}), Poincar\'e-Lindstedt (e.g., \cite{nayfeh2008perturbation, Verhulst:96}), and harmonic balance (e.g., \cite{nakhla1976piecewise}; note it can be viewed as a Galerkin method); however, it is nontrivial to adapt them to our problem. This is because, although the two oscillators' original intrinsic frequencies are 1:2, the nonlinearity may shift their frequencies differently such that no single time-transformation is sufficient for both oscillators any more. Therefore, we will follow the more convenient averaging approach, with an objective of again providing a simple improvement of its accuracy. For the relationship between averaging and the method of multiple scales, please see, e.g., \cite{perko1969higher}.

\subsection{The case of exactly resonant excitation}
\label{sec_CPUT_fixedPtExactlyResonant}
\paragraph{New averaged system.} If one is interested in an $\mathcal{O}(\epsilon)$ approximation of the solution till $\mathcal{O}(\epsilon^{-1})$ time, the $\mathcal{O}(\epsilon^2)$ terms in \eqref{eq_CPUToriginal} are unimportant, and the governing equations can be rewritten as
\begin{equation}
\begin{dcases}
	\dot{V}=U \\
	\dot{U}=-\omega^2 V - \epsilon \gamma U + \epsilon \alpha y V \\
	\dot{y}=z \\
	\dot{z}=-(2\omega)^2 y - \epsilon\beta z+\epsilon F\sin(2\omega t)+\epsilon\frac{V^2}{(D-y)^2}
\end{dcases}. \label{eq_CPUT_EOMexactFrequency}
\end{equation}
Viewing $X=[V,U,y,z]$, this is in the form of \eqref{eq_canonicalSystem}. However, an analytical expression for the average of $\epsilon\frac{V^2}{(D-y)^2}$ is not obtainable, and we thus make a reasonable assumption that the distance variation between the capacitor's double plates is much smaller than the distance itself (i.e., $|y| \ll D$), which permits a 1st-order Taylor approximation of the nonlinearity: $\epsilon\frac{V^2}{(D-y)^2} \approx \epsilon\left(\frac{V^2}{D^2}+\frac{2V^2 y}{D^3}\right)$.

To characterize the $\mathcal{O}(\epsilon^{-1})$ dynamics of \eqref{eq_CPUT_EOMexactFrequency}, we resort to the proposed coordinate transform $X\mapsto Z$ \eqref{eq_transformNew}. However, since oscillation amplitudes will be the quantities of interest for physical reasons (e.g., $V$ amplitude is the AC voltage), we'll further transform $Z$ to polar coordinates. Overall, we will use a coordinate transform from $[V,U,y,z]$ to $[\rho,\phi,r,\theta]$ via
\begin{equation}
\begin{dcases}
	V=\rho\cos(\omega(t+\phi)) \\
	U=-\omega\rho\sin(\omega(t+\phi)) \\
	y=r\cos(2\omega(t+\theta))+\frac{\epsilon\rho^2}{2D^2(2\omega)^2} \\
	z=-2\omega r\sin(2\omega(t+\theta))
\end{dcases}
\label{eq_CPUT_coordTransf}
\end{equation}
where the $\mathcal{O}(\epsilon)$ term is due to the local average $\epsilon P(Z(t))$ in \eqref{eq_transformNew} and crucial for an improved accuracy of averaging.

Such a coordinate transformation changes the dynamics according to
\begin{equation}
\begin{dcases}
	\dot{\rho}=\frac{\omega^2 V \dot{V}+U \dot{U}}{\omega^2 \rho} \\
	\dot{\phi}=\frac{\dot{V}U-\dot{U}V}{\omega^2 \rho^2}-1 \\
	\dot{r}=\frac{(2\omega)^2 \left(y-\frac{\epsilon \rho^2}{2D^2(2\omega)^2}\right) \left(\dot{y}-\frac{\epsilon \rho\dot{\rho}}{D^2(2\omega)^2}\right)+z \dot{z}}{(2\omega)^2 r} \\
	\dot{\theta}=\frac{\left(\dot{y}-\frac{\epsilon \rho \dot{\rho}}{D^2(2\omega)^2}\right)z-\dot{z}\left(y-\frac{\epsilon \rho^2}{2D^2(2\omega)^2}\right)}{(2\omega)^2 r^2}-1
\end{dcases}. 
\end{equation} 
Expressing the right hand sides in terms of $\rho,\phi,r,\theta$, it is easy to see they are $2\pi/(2\omega)$ or $2\pi/\omega$ periodic functions in $t$. In addition, $\rho,\phi,r,\theta$ are all slow variables because their time derivatives are $\mathcal{O}(\epsilon)$, and averaging approximation is applicable. For a simplification, note the $-\frac{\epsilon \rho\dot{\rho}}{D^2(2\omega)^2}$ terms are $\mathcal{O}(\epsilon^2)$, and thus the 1st-order averaged system is
\begin{equation}
\begin{dcases}
	\dot{\rho}=\left\langle \frac{\omega^2 V \dot{V}+U \dot{U}}{\omega^2 \rho} \right\rangle_t \\
	\dot{\phi}=\left\langle \frac{\dot{V}U-\dot{U}V}{\omega^2 \rho^2}-1 \right\rangle_t \\
	\dot{r}=\left\langle \frac{(2\omega)^2 \left(y-\frac{\epsilon \rho^2}{2D^2(2\omega)^2}\right) \dot{y} +z \dot{z}}{(2\omega)^2 r} \right\rangle_t \\
	\dot{\theta}=\left\langle \frac{\dot{y}z-\dot{z}\left(y-\frac{\epsilon \rho^2}{2D^2(2\omega)^2}\right)}{(2\omega)^2 r^2}-1 \right\rangle_t
\end{dcases} \label{eqn_CPUT_polarEOMavgGeneric} 
\end{equation} 
where \eqref{eq_CPUT_coordTransf} has to be substituted into \eqref{eq_CPUT_EOMexactFrequency}, and then both \eqref{eq_CPUT_coordTransf}, \eqref{eq_CPUT_EOMexactFrequency} plugged into \eqref{eqn_CPUT_polarEOMavgGeneric}, before time averaging over the period of $2\pi/\omega$ is performed. After some algebra, the averaged system \eqref{eqn_CPUT_polarEOMavgGeneric} is computed to be
\begin{equation}
\begin{dcases}
	\dot{\rho}= \frac{\epsilon}{4\omega}\rho \Big( -2\gamma\omega + r\alpha\sin\big(2(\theta-\phi)\omega\big) \Big) \\
	\dot{\phi}= -\frac{\epsilon\alpha}{16 D^2 \omega^4} \Big( \epsilon\rho^2 + 4D^2 r \omega^2 \cos\big(2(\theta-\phi)\omega\big) \Big) \\
	\dot{r}= -\frac{\epsilon}{32 D^5 \omega^3} \Big( 8D^5\omega^2 \big(2r\beta\omega + F\cos(2\theta\omega)\big) + \rho^2(\epsilon \rho^2+4D^3\omega^2)\sin\big(2(\theta-\phi)\omega\big) \Big) \\
	\dot{\theta}= -\frac{\epsilon}{64 D^5 \omega^4} \frac{1}{r} \Big( \rho^2(\epsilon \rho^2+4D^3\omega^2)\cos\big(2(\theta-\phi)\omega\big) + 8D^2\omega^2\big( r\rho^2-D^3 F \sin(2\theta\omega)\big) \Big)
\end{dcases} \label{eqn_CPUT_polarEOMavg} 
\end{equation}

\paragraph{Fixed point of the averaged system \eqref{eqn_CPUT_polarEOMavg}.}
The fixed point of \eqref{eqn_CPUT_polarEOMavg} is a real solution of
\begin{align}
	0&= -2\gamma\omega + r\alpha\sin\big(2(\theta-\phi)\omega\big) \label{eq_CPUT_fixedPt1} \\
	0&= \epsilon\rho^2 + 4D^2 r \omega^2 \cos\big(2(\theta-\phi)\omega\big) \label{eq_CPUT_fixedPt2} \\
	0&= 8D^5\omega^2 \big(2r\beta\omega + F\cos(2\theta\omega)\big) + \rho^2(\epsilon \rho^2+4D^3\omega^2)\sin\big(2(\theta-\phi)\omega\big) \label{eq_CPUT_fixedPt3} \\
	0&= \rho^2(\epsilon \rho^2+4D^3\omega^2)\cos\big(2(\theta-\phi)\omega\big) + 8D^2\omega^2\big( r\rho^2-D^3 F \sin(2\theta\omega)\big) \label{eq_CPUT_fixedPt4}
\end{align}
We now solve this transcendental algebraic system under the intrinsic constraint of $\rho\geq 0$, $r\geq 0$:
let $\sigma=\rho^2$ and note \eqref{eq_CPUT_fixedPt1} and \eqref{eq_CPUT_fixedPt2} give
\begin{align}
	& \sin\big(2(\theta-\phi)\omega\big) = \frac{2\gamma\omega}{r\alpha}, \label{eq_CPUT_Sm}  \\
	& \cos\big(2(\theta-\phi)\omega\big) = -\frac{\epsilon\sigma}{4D^2 \omega^2 r}. \label{eq_CPUT_Cm}
\end{align}
Their sum of squares being 1 leads to
\[
	r=\sqrt{\left(\frac{2\gamma\omega}{\alpha}\right)^2+\left(\frac{\epsilon\sigma}{4D^2 \omega^2}\right)^2}.
\]
Similarly, \eqref{eq_CPUT_fixedPt3} and \eqref{eq_CPUT_fixedPt4} give
\begin{align}
	& \sin\big(2\theta\omega\big) = \frac{(4 D^3 \sigma  \omega ^2+ \sigma ^2 \epsilon)\cos\big(2(\theta-\phi)\omega\big) +8 D^2 r \sigma  \omega ^2}{8 D^5 F \omega ^2}, \\
	& \cos\big(2\theta\omega\big) = \frac{-(4 D^3 \sigma  \omega ^2+\sigma ^2 \epsilon)\sin\big(2(\theta-\phi)\omega\big) -16 \beta  D^5 r \omega ^3}{8 D^5 F \omega ^2}.
\end{align}
Upon substituting \eqref{eq_CPUT_Sm} and \eqref{eq_CPUT_Cm}, $\sin^2(2\theta\omega)+\cos^2(2\theta\omega)=1$ leads to
\begin{align*}
	& \frac{1}{64D^{10} F^2\alpha^2\omega^4} (A_0+A_1\sigma+A_2\sigma^2+A_3\sigma^3+A_4\sigma^4)=0, \\
	& \text{where}~ \begin{dcases}
	  A_0=1024 \beta ^2 \gamma ^2 D^{10} \omega ^8-64 \alpha ^2 D^{10} F^2 \omega ^4 \\
	  A_1=256 \alpha  \beta  \gamma  D^8 \omega ^6 \\
	  A_2=16 \alpha ^2 \beta ^2 D^6 \omega ^2 \epsilon ^2+16 \alpha ^2 D^6 \omega ^4+64 \alpha  \beta  \gamma  D^5 \omega ^4 \epsilon +256 \gamma ^2 D^4 \omega ^6 \\
	  A_3=-8 \alpha ^2 D^3 \omega ^2 \epsilon \\
	  A_4=\alpha ^2 \epsilon ^2
	\end{dcases}
\end{align*}
This degree-four polynomial can be analytically solved; however, the solution expression is too lengthy to provide insights about its dependence on the parameters, thus not helpful for designing the CPUT system. Therefore, we will use singular perturbation method (e.g., \cite{kevorkian2012multiple}) to approximate its solutions instead.

More specifically, because of the scaling of $A_i$'s, let $B_0=A_0$, $B_1=A_1$, $B_2=A_2$, $\epsilon B_3=A_3$, $\epsilon^2 B_4=A_4$, and then the polynomial to solve becomes
\[
	B_0+B_1\sigma+B_2\sigma^2+\epsilon B_3\sigma^3+\epsilon^2 B_4\sigma^4=0,
\]
We are only interested in its non-negative real solution(s).

There can only be four solutions in total. First assuming $\sigma=\mathcal{O}(1)$ (the regular perturbation regime), one obtains
\begin{equation}
	B_0+B_1\sigma+B_2\sigma^2 \approx 0,
	\label{eq_CPUT_rootsRegular}
\end{equation}
whose roots will give two approximate solutions. Then assuming $\sigma=\mathcal{O}(\epsilon^{-1})$ (a singular perturbation regime) and letting $\sigma=\epsilon^{-1} s$, one can collect leading terms and obtain
\begin{equation}
	B_2 + B_3 s + B_4 s^2 \approx 0,
	\label{eq_CPUT_rootsSingular}
\end{equation}
whose roots will give two additional approximate solutions.

Roots of \eqref{eq_CPUT_rootsSingular} are complex because $B_3^2-4B_2B_4=-64D^2\alpha^2\omega^2((D\alpha\beta\epsilon+2\gamma\omega^2)^2+12\gamma^2\omega^4) < 0$. They are therefore irrelevant.

On the other hand, solving \eqref{eq_CPUT_rootsRegular} leads to (keeping the $\sigma\geq 0$ solution only):
\begin{align}
	& \sigma \approx \frac{-8 D^4 \alpha  \beta  \gamma  \omega ^4 + 2 D^3\omega\sqrt{N}}{D^2 \alpha ^2 \omega^2 + 16 \gamma ^2 \omega ^4 + 4 \alpha  \beta  \gamma  D \omega ^2 \epsilon + \alpha ^2 D^2 \beta ^2 \epsilon ^2}, \label{eq_CPUT_improvedFixedPtCore}  \\
	& \text{where}\quad N = D^2 F^2 \alpha^4 \omega^2 + 16 F^2 \alpha^2 \gamma ^2 \omega ^4 - 256 \beta ^2 \gamma ^4 \omega ^8 \\
	& \qquad\qquad + 4 \alpha  \beta  \gamma  D \omega ^2 \left(\alpha ^2 F^2-16 \beta ^2 \gamma ^2 \omega ^4\right) \epsilon 
	+ (D^2 F^2 \alpha ^4 \beta ^2 - 16 D^2 \alpha ^2 \beta ^4 \gamma ^2 \omega ^4) \epsilon ^2. \nonumber
\end{align}
The corresponding $\rho$ and $r$ are:
\begin{equation} \bm{
	\rho = \sqrt{\sigma}, \qquad
	r = \sqrt{\left(\frac{2\gamma\omega}{\alpha}\right)^2+\left(\frac{\epsilon \rho^2}{4 D^2 w^2}\right)^2},
	}\label{eq_CPUT_improvedFixedPt}
\end{equation}
which correspond to	steady-state oscillation amplitudes of $V$ and $y$.

\begin{Remark}
	When necessary, one can simplify \eqref{eq_CPUT_improvedFixedPtCore} to leading order as
	\[
		\sigma \approx \frac{-8 D^4 \alpha \beta \gamma \omega^2 + 2D^3\sqrt{D^2 F^2 \alpha^4+16F^2\alpha^2\gamma^2\omega^2-256\beta^2\gamma^4\omega^6}}{D^2\alpha^2+16\gamma^2\omega^2}.
	\]
	Experimentalists may find this expression easier to use when tuning the system parameters for a desired output voltage; this approximation is less accurate though.
\end{Remark}

\paragraph{Remarks on dynamics of the averaged system \eqref{eqn_CPUT_polarEOMavg}.}
At least for the parameters used in \cite{CPUT2018IEEE}, the fixed point identified above can be verified by linearization to be asymptotic stable (i.e., a sink), and hence it corresponds to a steady state of the CPUT. 

However, although the system in its physical domain ($\rho\geq 0$, $r\geq 0$, $\phi\in\mathbb{T}$, $\theta\in \mathbb{T}$) contains only one sink, it is unclear whether there are higher-order attractors such as limit cycle or even chaotic attractor (note the system is more than 2 dimensional and thus Poincar\'e-Bendixson theorem does not apply). Therefore, the possibility that some initial condition not converging to the identified steady state has not been theoretically ruled out.

To better understand whether the sink is the global attractor, we performed a brute force numerical search in the 4-dimensional state-space. Initial conditions were sampled from regular grid points and after sufficient long time they all converged to the identified sink. This suggests that, if there were any attractor in the averaged system \ref{eqn_CPUT_polarEOMavg} other than the sink, its geometrical shape must be rather irregular because it avoided all the sampled grid points.

More specifically, for parameters
\begin{align*}
	& \epsilon = 0.069908094621482, ~
	D = 12, ~
	\omega = 0.628318530717959, ~
	F = 4.517732098486560, \\
	& \alpha = 0.471019510657106, ~
	\beta = 3.388299073864920, ~
	\gamma = 0.208520337367901,
\end{align*}
we used the Georgia Tech high performance computing cluster (PACE) to simulate 71,680,000 initial conditions $[r,\rho,\phi,\theta](0) \in [0.01:0.01:2]\times[0.1:0.1:35]\times[0:0.2:2\pi)\times[0:0.2:2\pi)$, i.e., drawn from a uniform grid. For each initial condition, the averaged system is numerically integrated by 4th-order Runge-Kutta with $h=0.01$ till $T=450$, and the 2-norm of the vector field at the last step as well as the final 2-distance to the fixed point [24.442613175475309;   1.077007670858842;   0.585849324582913;  -2.419228080303699] were recorded. The maxima of these two quantities over all initial conditions are, respectively, $5.6650\times 10^{-14}$ and $1.2397\times 10^{-10}$, suggesting convergence to the fixed point in all cases.

Finally, it is important to recall that attractors of the averaged system may not have correspondences in the original system \eqref{eq_CPUT_EOMexactFrequency}, and vice versa. In fact, the averaged system may be an accurate approximation only till $\mathcal{O}(\epsilon^{-1})$. Infinite time approximations of dynamical structures trace back to at least KAM theory (e.g., \cite{arnold1997mathematical, moser1973stable, poschel1982integrability} and references therein) and still remain as a research frontier.

\paragraph{Numerical demonstration.} One simplest difference produced by the classical averaging (Section \ref{sec_classicalAveraging}) and the improved one (Section \ref{sec_improvedMethod}) for this problem \eqref{eq_CPUT_EOMexactFrequency} is, whether the mean of $y$ oscillations at the steady state is zero.

More specifically, one can perform the classical averaging by first using the coordinate transform from $[V,U,y,z]$ to $[\hat{\rho},\hat{\phi},\hat{r},\hat{\theta}]$ given by
\[
\begin{dcases}
	V=\hat{\rho}\cos(\omega(t+\hat{\phi})) \\
	U=-\omega\hat{\rho} \sin(\omega(t+\hat{\phi})) \\
	y=\hat{r}\cos(2\omega(t+\hat{\theta})) \\
	z=-2\omega \hat{r}\sin(2\omega(t+\hat{\theta}))
\end{dcases},
\]
and then averaging the $\hat{\rho},\hat{\phi},\hat{r},\hat{\theta}$ system and analyzing its fixed point (details omitted). However, that fixed point will correspond to a steady-state where $y$ oscillates with amplitude $\hat{r}$ about a mean of zero. This should be contrasted with the prediction of improved averaging, which, according to \eqref{eq_CPUT_coordTransf}, has its steady-state in which $y$ oscillates with amplitude $r$ about a mean of $\epsilon \rho^2 / (8D^2\omega^2)$.

\begin{figure}[h]
\centering
\footnotesize
\subfigure[Benchmark solution of \eqref{eq_CPUT_EOMexactFrequency} obtained by fine numerical simulation. Green line corresponds to $y=0$ and shows the asymmetry of the $y$ trajectory.
]{
\includegraphics[width=0.48\textwidth]{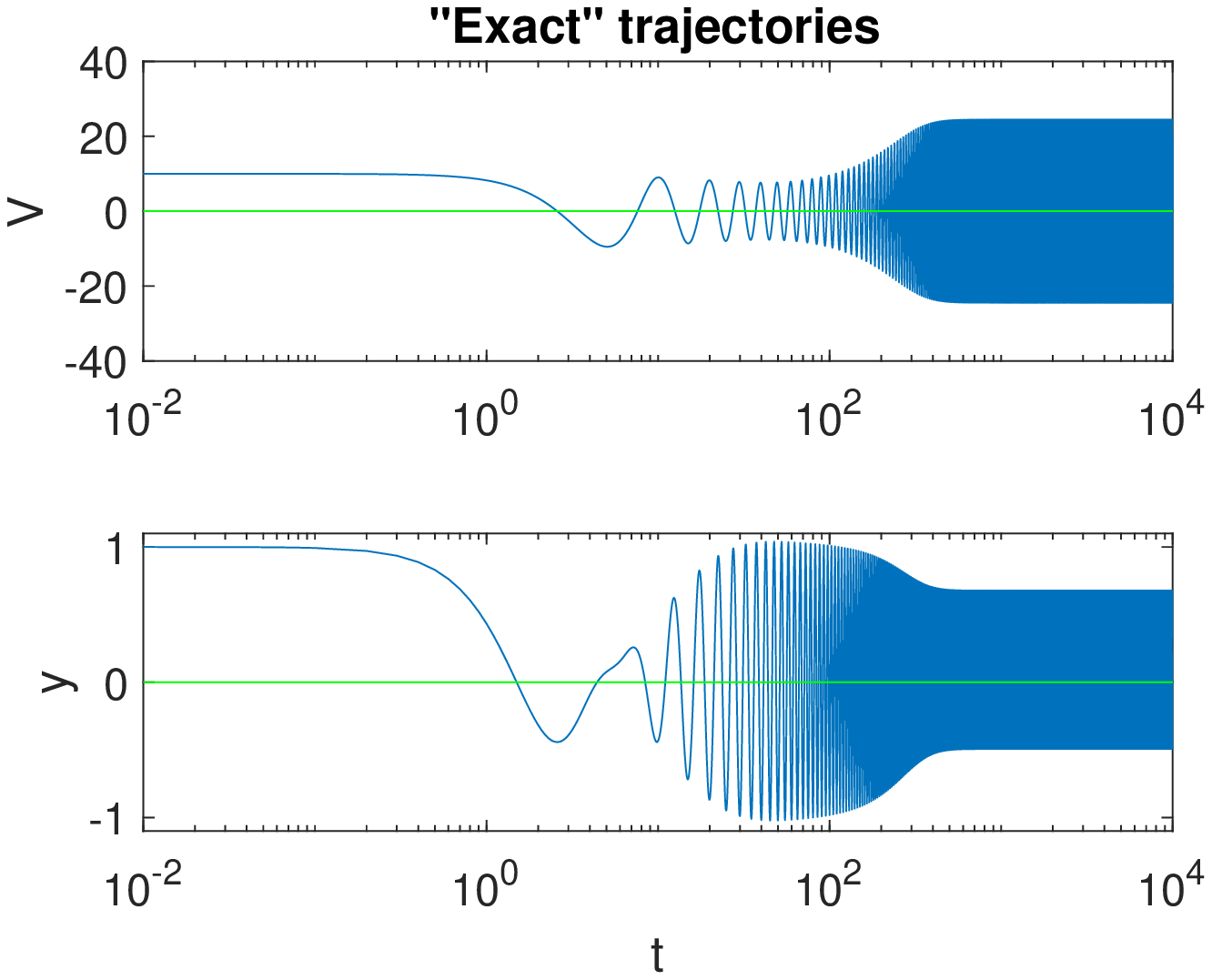}
\label{fig_CPUT_ExactTrajectories}
}
\subfigure[Deviations of trajectories obtained by fine numerical simulations of the classically averaged system and the improved averaged system \eqref{eqn_CPUT_polarEOMavg} from the benchmark.]{
\includegraphics[width=0.48\textwidth]{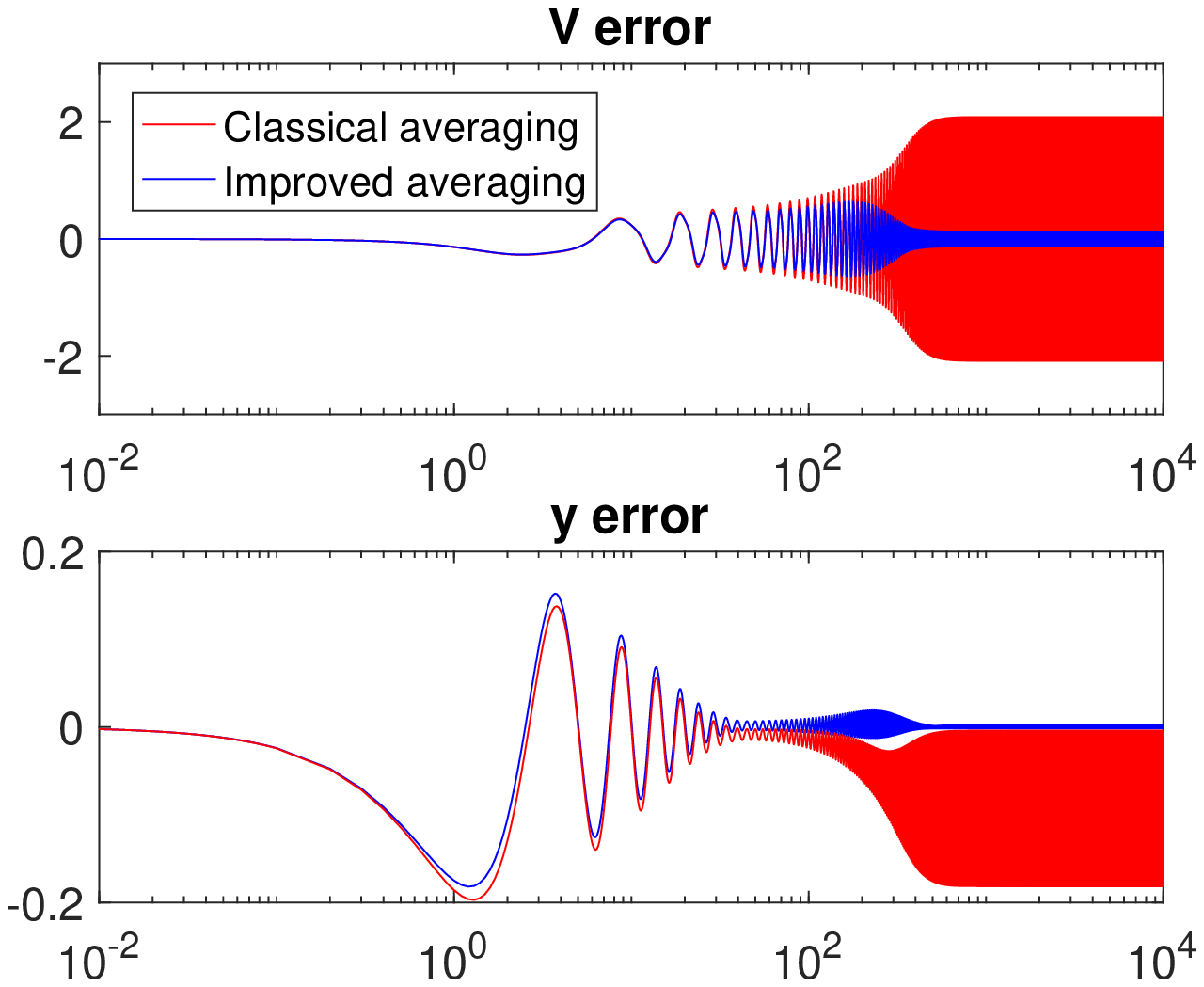}
\label{fig_CPUT_numericalErrors}
}
\caption{\footnotesize Performances of classical and improved averages for CPUT. {$[V,\dot{V},y,\dot{y}](0)=[10,0,1,0]$.} Note the time axis is in log scale for showing both the nontrivial transient and the steady state.}
\end{figure}

\begin{table}[h]
  \begin{tabular}{r | c | c | c | c}
  	& $V$ amplitude & $V$ mean & $y$ amplitude & $y$ mean \\
  	\hline \hline
	Truth (fine numerical solution) &
	$\approx 24.4409$ & $\approx 0.0000$ & $\approx 0.58481$ & $\approx 0.091952$ \\
	Classical averaging prediction &
	$\approx 24.\textcolor{red}{8527}$ & $0$ & $\approx 0.5\textcolor{red}{5631}$ & $=0.0\textcolor{red}{00000}$ \\
	Improved averaging prediction A &
	$\approx 24.44\textcolor{red}{26}$ & $0$ & $\approx 0.585\textcolor{red}{85}$ & $\approx 0.091\textcolor{blue}{8}\textcolor{red}{35}$ \\
	Improved averaging prediction B &
	$\approx 24.\textcolor{blue}{3}\textcolor{red}{626}$ & $0$ & $\approx 0.585\textcolor{red}{47}$ & $\approx 0.09\textcolor{red}{1235}$
  \end{tabular}
  \caption{\footnotesize CPUT's steady-state oscillation parameters. Improved averaging prediction A corresponds to the numerical solution of fixed point equations \eqref{eq_CPUT_fixedPt1}-\eqref{eq_CPUT_fixedPt4}. Improved averaging prediction B corresponds to the analytically approximated fixed point \eqref{eq_CPUT_improvedFixedPtCore}-\eqref{eq_CPUT_improvedFixedPt}.}
  \label{tab_CPUT}
\end{table}

Figure \ref{fig_CPUT_ExactTrajectories} shows how the solution of the original problem \eqref{eq_CPUT_EOMexactFrequency} transits into the steady-state of harmonic oscillations. Important to note is $y$ oscillates around a nonzero mean. The proposed approach well captures this, but the classical averaging does not (see Table \ref{tab_CPUT}). This improved accuracy can be quantitatively seen from the significantly reduced error of the former (Figure \ref{fig_CPUT_numericalErrors}).

\subsection{The case of inexact resonant excitation: the dependence on the excitation frequency and amplitude}
We now quantify how strong the external excitation needs to be and to what extent the perturbation frequency can deviate from $2\omega$, in order for the system to maintain its nonzero steady-state (which means a steady output of energy). To do so, consider, instead of \eqref{eq_CPUT_EOMexactFrequency}, the following:
\[ \begin{dcases}
	\dot{V}=U \\
	\dot{U}=-\omega^2 V - \epsilon \gamma U + \epsilon \alpha y V \\
	\dot{y}=z \\
	\dot{z}=-(2\omega)^2 y - \epsilon\beta z+\epsilon F\sin(2(1\textcolor{blue}{+\epsilon\Delta}) \omega t)+\epsilon\frac{V^2}{(D-y)^2}
\end{dcases}, \]
where $\Delta$ quantifies the deviation of frequency. To analyze the system, introduce a new time $\tau=(1+\epsilon \Delta)t$ and denote $d/d\tau$ by prime. Then
\[	\begin{dcases}
	(1+\epsilon\Delta)V'=U \\
	(1+\epsilon\Delta)U'=-\omega^2 V - \epsilon \gamma U + \epsilon \alpha y V \\
	(1+\epsilon\Delta)y'=z \\
	(1+\epsilon\Delta)z'=-(2\omega)^2 y - \epsilon\beta z+\epsilon F\sin(2\omega \tau)+\epsilon\frac{V^2}{(D-y)^2}
\end{dcases}. \]
Discard $o(\epsilon)$ terms, because similar to before they will not affect an $\mathcal{O}(\epsilon)$ approximation till at least $\tau=\mathcal{O}(\epsilon^{-1})$. We then obtain
\begin{equation}
\begin{dcases}
	V'=U - \epsilon\Delta U \\
	U'=-\omega^2 V + \epsilon \Delta \omega^2 V - \epsilon \gamma U + \epsilon \alpha y V \\
	y'=z - \epsilon\Delta z \\
	z'=-(2\omega)^2 y + \epsilon\Delta(2\omega)^2 y - \epsilon\beta z+\epsilon F\sin(2\omega \tau)+\epsilon\frac{V^2}{(D-y)^2}
\end{dcases}. \label{eq_CPUT_EOMresonantTongue}
\end{equation} 
The same coordinate transformation (\eqref{eq_CPUT_coordTransf}, with $t$ replaced by $\tau$), the same approximation of $\epsilon\frac{V^2}{(D-y)^2} \approx \epsilon\left(\frac{V^2}{D^2}+\frac{2V^2 y}{D^3}\right)$, and the same omission of $-\frac{\epsilon \rho\dot{\rho}}{D^2(2\omega)^2}$ lead to a system that averages (over $\tau$) to
\begin{equation}
\begin{dcases}
	\rho'= \frac{\epsilon}{4\omega}\rho \Big( -2\gamma\omega + r\alpha\sin\big(2(\theta-\phi)\omega\big) \Big) \\
	\phi'= -\frac{\epsilon}{16 D^2 \omega^4} \Big( \epsilon \alpha \rho^2 {\textcolor{blue}{+16 D^2\omega^4\Delta}} + 4D^2 \alpha r \omega^2 \cos\big(2(\theta-\phi)\omega\big) \Big) \\
	r'= -\frac{\epsilon}{32 D^5 \omega^3} \Big( 8D^5\omega^2 \big(2r\beta\omega + F\cos(2\theta\omega)\big) + \rho^2(\epsilon \rho^2+4D^3\omega^2)\sin\big(2(\theta-\phi)\omega\big) \Big) \\
	\theta'= -\frac{\epsilon}{64 D^5 \omega^4} \frac{1}{r} \Big( \rho^2(\epsilon \rho^2+4D^3\omega^2)\cos\big(2(\theta-\phi)\omega\big) + 8D^2\omega^2\big( r (\rho^2 {\textcolor{blue}{+8D^3\omega^2\Delta}}) -D^3 F \sin(2\theta\omega)\big) \Big)
\end{dcases}, \label{eqn_polarEOMresonantTongueAvg} 
\end{equation}
where blue are terms not in the previously averaged system \eqref{eqn_CPUT_polarEOMavg}.

Now the task is to analyze the fixed point(s) of \eqref{eqn_polarEOMresonantTongueAvg}.

Similar to before, $\sin^2\big(2(\theta-\phi)\omega\big)+\cos^2\big(2(\theta-\phi)\omega\big)=1$ leads to
\[
	r=\frac{1}{\alpha} 
	\sqrt{\left(2\gamma\omega\right)^2+\left(\frac{\epsilon \alpha \rho^2+16D^2\omega^4\Delta}{4 D^2 w^2}\right)^2}.
\]
$\rho$ can also be exactly obtained or approximated in an analogous way (see Section \ref{sec_CPUT_fixedPtExactlyResonant}).

For the sake of a concise expression, however, from here on we will only construct a leading order in $\epsilon$ approximation of the fixed point, by solving the following approximation of the fixed point equation:
\begin{align*}
	&0= -2\gamma\omega + r\alpha \sin\big(2(\theta-\phi)\omega\big) \\
	&0= 4\omega^2\Delta + r\alpha \cos\big(2(\theta-\phi)\omega\big) \\
	&0= 8D^5\omega^2 \big(2r\beta\omega + F\cos(2\theta\omega)\big) + 4D^3\omega^2\rho^2\sin\big(2(\theta-\phi)\omega\big) \\
	&0= 8D^2\omega^2\big( (\rho^2+8D^3\omega^2\Delta)r-D^3 F \sin(2\theta\omega)\big) + 4D^3\omega^2\rho^2\cos\big(2(\theta-\phi)\omega\big) 
\end{align*}
Repeated applications of $\sin^2(\cdot)+\cos^2(\cdot)=1$ like before show that the only positive real solution for $\rho^2$ is
\begin{align}
	\rho^2 = \frac{2 D^3 \left(\sqrt{\zeta}-4 \alpha  \beta  \gamma  D \omega ^2+32 \Delta  \omega ^4 \left(\Delta  \left(\alpha  D-8 \Delta  \omega ^2\right)-2 \gamma ^2\right)\right)}{16 \gamma ^2 \omega ^2+\left(\alpha  D-8 \Delta  \omega ^2\right)^2}
	\label{eq_CPUT_mixedSlnTongue}
\end{align}
where $\zeta$ is a lengthy expression:
\begin{align*}
	 \zeta = &-4096 \beta ^2 \Delta ^4 \omega ^{10}+1024 \beta  \Delta ^2 \omega ^8 \left(\alpha  D \Delta  (\beta +2 \gamma )-2 \beta  \gamma ^2\right)
	 -64 \omega ^6 \left(\alpha  D \Delta  (\beta +2 \gamma )-2 \beta  \gamma ^2\right)^2 \\
	 &\qquad\qquad +64 \alpha ^2 \Delta ^2 F^2 \omega ^4 +16 \alpha ^2 F^2 \omega ^2 \left(\gamma ^2-\alpha  D \Delta \right)+\alpha ^4 D^2 F^2.
\end{align*}
Due the its length, this result is not insightful yet, but we will see that it helps characterize the boundary of nontrivial steady-state solution in the parameter space. In fact, the non-negativity of the right hand side of \eqref{eq_CPUT_mixedSlnTongue} places a requirement on the perturbation strength $F$. More precisely, since the right hand side is in the form of $C_1+\sqrt{C_2 F^2+C_3}$ and thus monotonic with $F^2$, we solve $\rho^2=0$ with \eqref{eq_CPUT_mixedSlnTongue} to get a critical value of $F$:
\begin{equation} \bm{
	F^*=\frac{4\omega^2}{\alpha}\sqrt{\gamma^2+4\Delta^2\omega^2}\sqrt{\beta^2+16\Delta^2\omega^2}.
	}\label{eq_CPUT_excitationThreshold}
\end{equation}
The requirement for a nontrivial steady-state is $|F|\geq F^*$.

Note $F^*$ has a clean expression despite of the complicated intermediate results such as \eqref{eq_CPUT_mixedSlnTongue}. This expression predicts, for instance, that when the excitation frequency exactly matches the intrinsic frequency (the case of Section \ref{sec_CPUT_fixedPtExactlyResonant}), the excitation strength has to be at least $4\omega^2\gamma\beta/\alpha$ in order for CPUT to work. If the frequencies don't exactly match but deviate by only $\mathcal{O}(\epsilon)$, then the system can still operate, but the threshold on the excitation strength increases in a specific way that depends on the frequency deviation percentage $\Delta$.

\eqref{eq_CPUT_excitationThreshold} well matches the operational boundary numerically obtained in \cite{surappa2017capacitive}. The region of $|F|\geq F^*$ is in the same spirit as a resonance tongue (e.g., \cite{arnol1983remarks}, \cite{broer2000resonance} and references therein), except the latter is historically speaking for a 2-dimensional non-autonomous system but the CPUT considered here is 4-dimensional non-autonomous.

\section{Demonstration 2: Fermi-Pasta-Ulam problem}
\label{sec_FPU}
This section illustrates the improved accuracy of the proposed approach on the Fermi-Pasta-Ulam (FPU) system \cite{FPU:55}. FPU is chosen simply because it is a classical test problem, and we do not claim any new understanding of this profound and extensively investigated problem; instead, we refer to \cite{Fo92, rink2001symmetry, berman2005fermi, FlIvKa05, Hairer:04} for an incomplete list of discussions.

\begin{figure} [h]
\begin{tabular}{c}
\includegraphics[width=\textwidth]{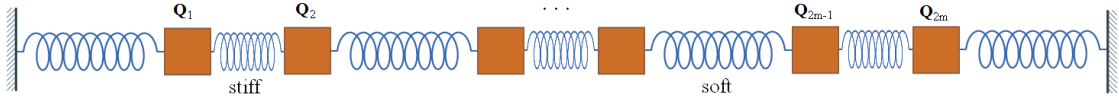}
\end{tabular}
\caption{\footnotesize Fermi-Pasta-Ulam problem: a 1D chain of point masses (exaggerated in the plot) connected alternatively via harmonic stiff and anharmonic soft springs}
\label{fig_FPUmodel}
\end{figure}

\paragraph{Setup.} The FPU problem, illustrated by Figure \ref{fig_FPUmodel}, is a mechanical system governed by the Hamiltonian
\[
    H(Q,P):=\frac{1}{2} \sum_{i=1}^m (P_{2i-1}^2+P_{2i}^2)+\frac{\omega^2}{4} \sum_{i=1}^m (Q_{2i}-Q_{2i-1})^2+ \sum_{i=0}^m (Q_{2i+1}-Q_{2i})^4,
\]
where $Q_1,\cdots,Q_{2m}$ are point mass displacements, $P_i$'s are their corresponding momenta, $Q_0=Q_{2m+1}=0$ are fixed (not part of the $2m$ degrees of freedom), and $\omega \gg 1$.

To better separate the slow and fast motions, the standard approach is to introduce the following canonical transformation
\begin{align*}
	& q_i = (Q_{2i}+Q_{2i-1})/\sqrt{2},	\qquad	q_{m+i}=(Q_{2i}-Q_{2i-1})/\sqrt{2}, \\
	& p_i = (P_{2i}+P_{2i-1})/\sqrt{2},	\qquad	p_{m+i}=(P_{2i}-P_{2i-1})/\sqrt{2},
\end{align*}
and then the transformed system is governed by the Hamiltonian
\begin{align*}
	& H(q,p)=\frac{1}{2} \sum_{i=1}^{2m} p_i^2 + \frac{\omega^2}{2}\sum_{i=m+1}^{2m} q_i^2 + \\
	& \qquad \frac{1}{4} \left( (q_1-q_{m+1})^4 + \sum_{i=1}^{m-1}((q_{i+1}-q_{m+i+1})-(q_i-q_{m+i}))^4 + (q_m+q_{2m})^4 \right).
\end{align*}
FPU is a classical weakly nonlinear system and it exhibits diverse behaviors over at least three timescales (see e.g., \cite{Hairer:04} for a concise review). Either from the Hamiltonian or by physical intuition, it is not difficult to see that (i) each stiff spring's displacement ($q_i$ for $i=m+1,\cdots,2m$) behaves like a harmonic oscillator at $\mathcal{O}(\omega^{-1})$ timescale, and (ii) each center of masses linked by a stiff spring ($q_i$ for $i=1,\cdots,m$) changes more slowly at a timescale of $\mathcal{O}(1)$. In addition, it was known that (iii) the energy exchange among stiff springs is associated with a third timescale of $\mathcal{O}(\omega)$: denote the energy of the $i^{th}$ stiff spring by
\begin{equation}
	I_i=\frac{1}{2}(p_{m+i}^2+\omega^2 q_{m+i}^2),\qquad i=1,\cdots,m
	\label{eq_FPU_adiabatic}
\end{equation}
then $I_i$'s start exchanging values with each other at this timescale. (iv) This energy exchange actually extends to even slower timescales, in fashions that could at least be periodic, quasiperiodic or chaotic (e.g., \cite{Fo92, rink2001symmetry, berman2005fermi, FlIvKa05}). (v) On the other hand, the total energy of the stiff springs is only an $\mathcal{O}(\omega^{-1})$ deviation from a constant.

FPU is another example of coupled oscillators and it again suits the proposed method. In fact, let $q_{\text{slow}}=[q_1,\cdots,q_m]^T$, $p_{\text{slow}}=[p_1,\cdots,p_m]^T$, $q_{\text{fast}}=[q_{m+1},\cdots,q_{2m}]^T$, $p_{\text{fast}}=[p_{m+1},\cdots,p_{2m}]^T$, then the governing dynamics given by Hamilton's equations $q_i'=\partial H / \partial p_i$ and $p_i'=-\partial H / \partial q_i$ are
\[ \begin{dcases}
	q'_{\text{slow}} &= p_{\text{slow}} \\
	p'_{\text{slow}} &= f \\
	q'_{\text{fast}} &= p_{\text{fast}} \\
	p'_{\text{fast}} &= -\omega^2 q_{\text{fast}} + g
\end{dcases} \]
for some functions $f$ and $g$. Now, rescale momentum by $p_{\text{fast}}=\omega v_{\text{fast}}$, $p_{\text{slow}}=v_{\text{slow}}$, and then slow down time by a factor of $\omega$, we have
\[\begin{dcases}
	\dot{q}_{\text{slow}} &= \frac{1}{\omega} v_{\text{slow}} \\
	\dot{v}_{\text{slow}} &= \frac{1}{\omega} f \\
	q'_{\text{fast}} &= v_{\text{fast}} \\
	v'_{\text{fast}} &= -q_{\text{fast}} + \frac{1}{\omega^2} g
\end{dcases}. \]
Therefore, FPU can be put in the form of \eqref{eq_canonicalSystem} with $\epsilon=1/\omega$. Note there $\Omega$ is $4m$-by-$4m$, skew-Hermitian, but its rank is only $2m$.

\begin{figure} [h]
\begin{tabular}{c}
\includegraphics[width=\textwidth]{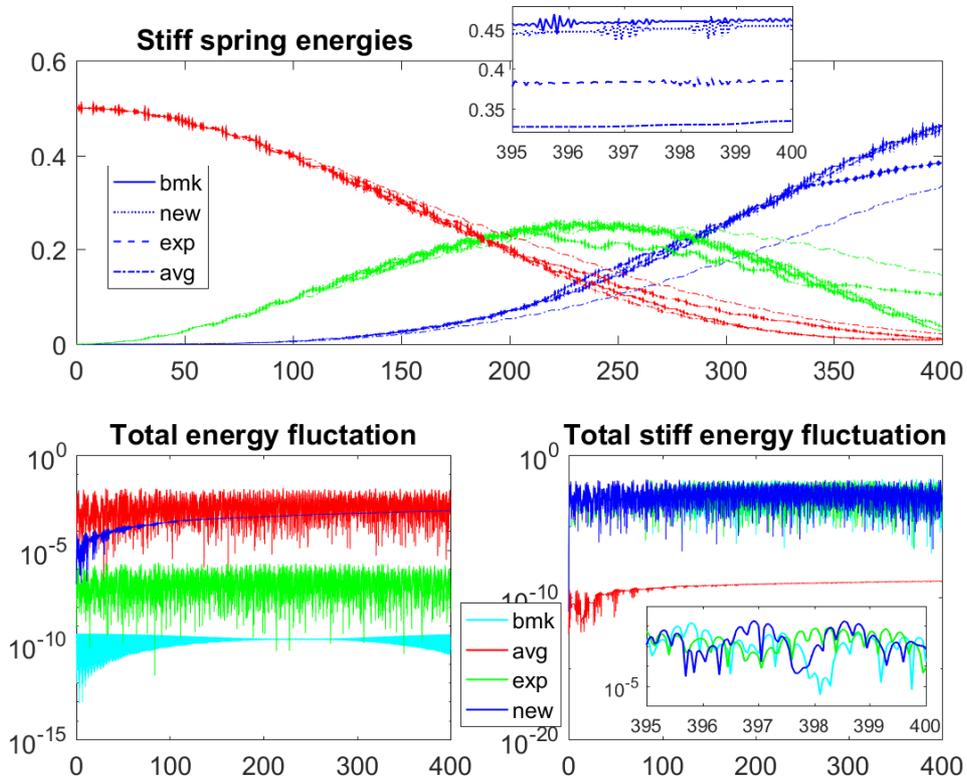}
\end{tabular}
\vspace{-20pt}
\caption{\footnotesize Comparison of classical averaging, improved averaging (denoted by `new'), and fine exponential integration of FPU with a benchmark solution.}
\label{fig_FPU_energyExchange}
\end{figure}

\begin{figure}[h]
\centering
\footnotesize
\subfigure[Errors in slow and fast positions and momenta]{
\includegraphics[width=0.5\textwidth]{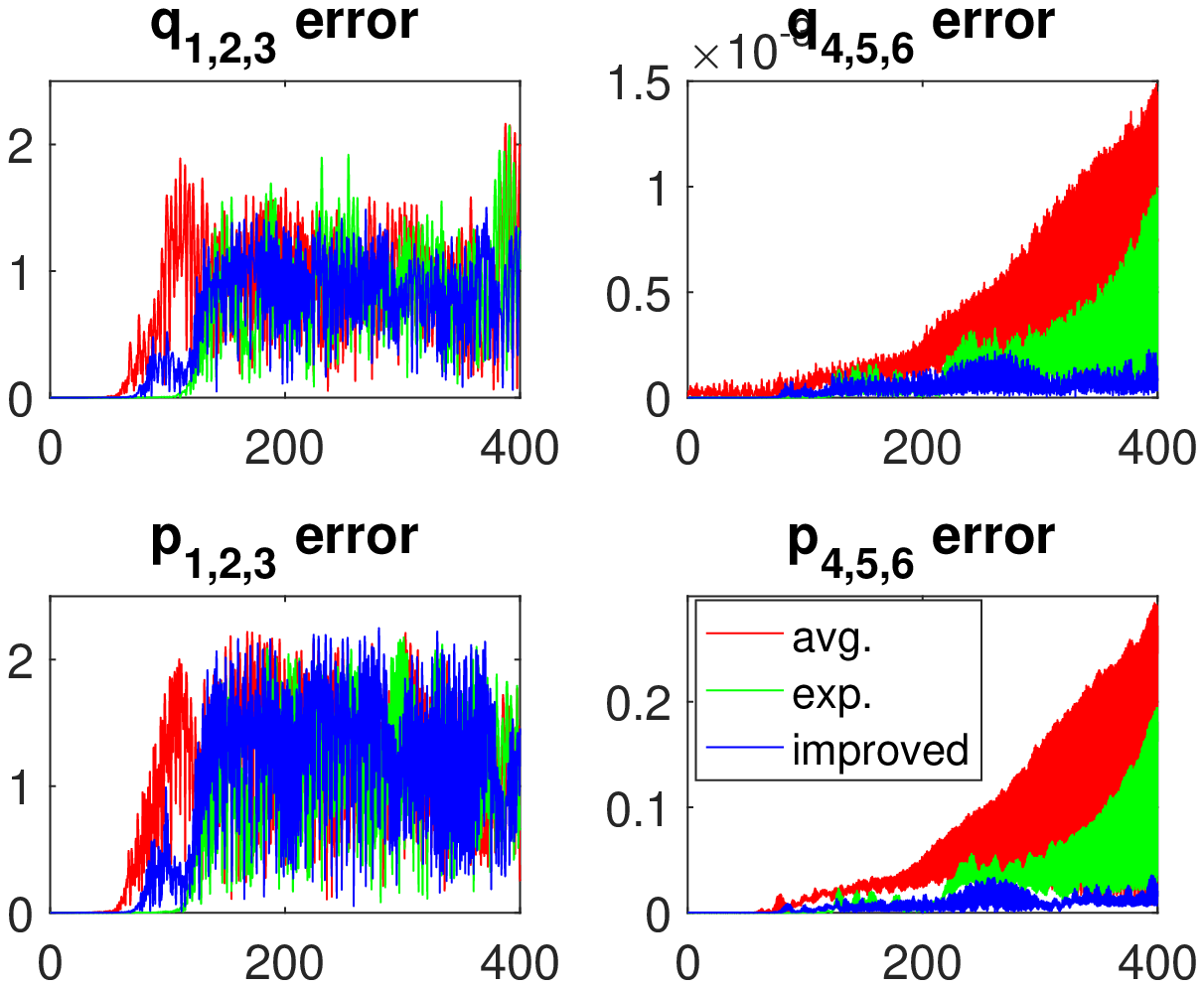}
\label{fig_FPU_qpErrors}
}
\hspace{-20pt}
\subfigure[Errors in stiff spring energies]{
\includegraphics[width=0.5\textwidth]{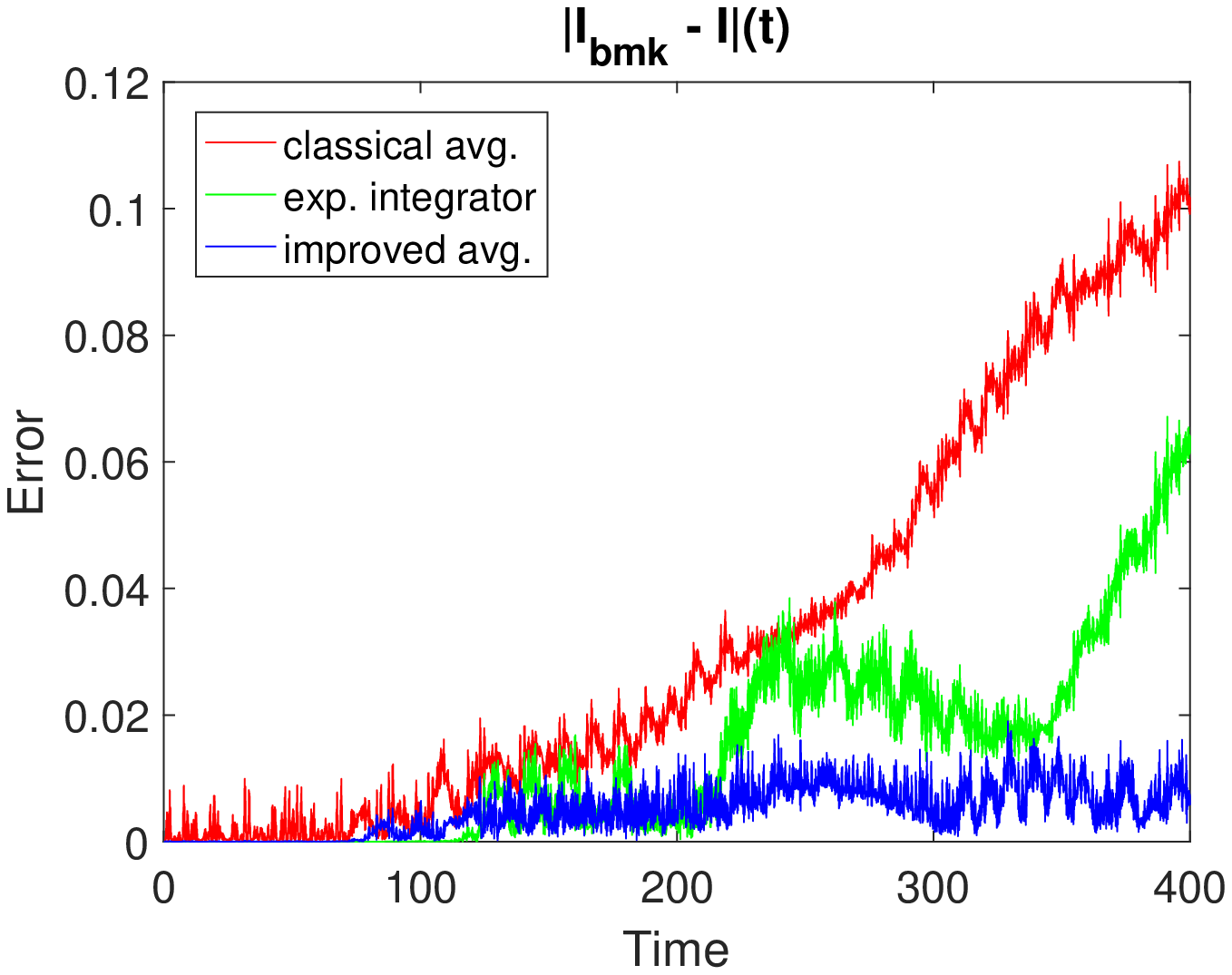}
\label{fig_FPU_Ierrors}
}
\caption{\footnotesize Absolute errors of classical averaging, improved averaging, and fine exponential integration of FPU (measured in the norm of 3-vector)}
\label{fig_FPU_errors}
\end{figure}

\paragraph{Numerical demonstration.} We now show that, in practice, the new averaging method can improve the accuracy of classical averaging beyond the theoretically justified $\mathcal{O}(\epsilon^{-1})$ timescale (i.e., $\mathcal{O}(1)$ time in plotted figures due to previously introduced time rescaling), and this improvement actually extends to $\mathcal{O}(\epsilon^{-2})$ timescale (i.e. $\mathcal{O}(\omega)$ timescale in the figures).

The demonstration will be based on the following experiment: $m=3$, $\omega=200$, simulation time $T=2\omega$, $q(0)=[0,0,0,0,0,0]^T$, $p(0)=[2,0,0,1,0,0]^T$. Classical averaging and improved averaging were based on 4th-order Runge-Kutta integration with $h=0.05$, in which the averaged vector fields (\eqref{eq_oldAveragedVectorField}, \eqref{eq_newAveragedVectorField} and \eqref{eq_CZ}) are computed on the fly via numerical averaging \eqref{eq_numAveragingPeriodic} with $N=10$. Benchmark solution for illustration and error quantification was obtained by a 4th-order symplectic integrator (see e.g., \cite{Hairer:04} or \cite{Tao2016PRE}) with very small timestep of $h=0.01/\omega$. Fine exponential integration was also performed as a side reference, based on a 2nd-order symmetric scheme that only uses one force evaluation per step, with $h=0.05/\omega=0.00025$.

Figures \ref{fig_FPU_energyExchange} and \ref{fig_FPU_errors} illustrate the followings:
\begin{itemize}
\item
	Upper panel of Figure \ref{fig_FPU_energyExchange} demonstrates better accuracy in $I_{1,2,3}$ \eqref{eq_FPU_adiabatic} of the new approach than that of classical averaging and exponential integration (with 200x smaller step size).
\item
	The system is autonomous Hamiltonian and the exact solution should conserve the total energy. Lower left panel of Figure \ref{fig_FPU_energyExchange} focuses on the accuracy of capturing this conservation. Note the benchmark and the exponential integrator are symplectic and thus intrinsically good at energy conservation (see e.g., \cite{Hairer:04} for theory). The averaged dynamics, however, may lose its Hamiltonian structure and were integrated by a non-symplectic integrator (see \cite{FLAVOR10} for ideas for a possible remedy). The improved averaging reduced the artificial energy deviation in classical averaging by magnitudes.
\item
	Lower right panel of Figure \ref{fig_FPU_energyExchange} focuses on the total energy of stiff springs. This quantity, $\sum_{i=1}^m I_i$, is known to be almost a constant, however admitting $\mathcal{O}(1/\omega)$ fluctuations. The classical averaging smoothed out this fluctuation incorrectly, while the improved averaging captured the same amount of fluctuation as the benchmark and the exponential integrator.
\item
	Three columns in Figure \ref{fig_FPU_errors} respectively illustrate deviations from the benchmark in the soft spring positions and momenta, the stiff spring positions and momenta, and the stiff spring energies. Recall the dynamics of these quantities span a wide range of timescales (respectively at $\mathcal{O}(1)$, $\mathcal{O}(1/\omega)$ and $\mathcal{O}(\omega)$ scales). Curiously, although the improved method is more accurate on all these quantities, it is more effective for observables of the stiff springs.
\end{itemize}

\begin{figure} [h]
\centering
\includegraphics[width=0.5\textwidth]{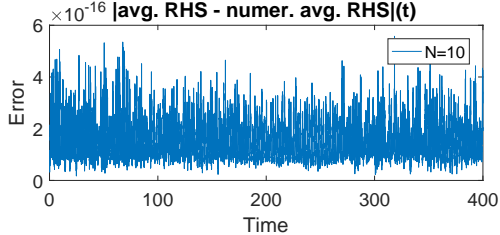}
\caption{\footnotesize Deviations of numerically averaged vector field from the exact average, evaluated in 2-norm along the simulation of classically averaged FPU.}
\label{fig_FPU_numAvgErr}
\end{figure}

In addition, Figure \ref{fig_FPU_numAvgErr} illustrates the high accuracy of numerical averaging \eqref{eq_numAveragingPeriodic}. Near machine precision was obtained using only 10 sample time points per averaged vector field evaluation. Lengthy expressions of the exact averaged vector field can be found in the Appendix.

It needs to be noted that numerical averaging of the vector field does introduce a computational overhead; however, it leads to the ability of using less time steps (see Section \ref{sec_numAvg}). Although the focus of this article is a simple analytical improvement of 1st-order averaging, not numerical averaging, time counts will still be provided for a qualitative (not quantitative) understanding of a curious reader: for the above experiment on an Intel i7-4600 laptop with Windows 7 x64 and MATLAB R2016b, the benchmark, the exponential integrator, classical averaging and the improved averaging respectively spent 99.36, 17.31, 5.98 and 12.11 seconds. Note this is by no means a claim against exponential integrators: we expect improved accuracy if a 4th-order exponential integrator with small step size were used instead; in fact, our approach with $N=1$ will become an exponential integrator (Section \ref{sec_numAvg}). However, as we aimed to quantify the error of the averaged dynamics, high-order method was employed for its integration to prevent the quantification from being polluted by the numerical integration error. The exponential integration, on the hand, only has numerical error, and making it vanish will only produce another benchmark; therefore, we chose a 2nd-order version simply to provide a middle ground.

\section{Demonstration 3: (2+1)-dimensional advection equation with weak nonlinear reaction}
\label{sec_PDE}
\subsection{The setup}
\label{sec_PDE_setup}
Methods in Section \ref{sec_method} generalize to infinite-dimensional problems too. Consider the semilinear initial value problem
\begin{equation}
	\partial_t u = L u + \epsilon f(u,z,t)
	\label{eq_canonicalPDE}
\end{equation}
equipped with initial condition $u(z,0)=u_0(z)$ and appropriate boundary conditions, where $L$ is a skew-Hermitian linear spatial differential operator, and $u(z,t)$ is a vector-valued function of multidimensional spatial coordinate $z$ and scalar time $t$. Here both the linear operator $L$ and the nonlinearity $f$ can be non-local in space, e.g., $(\partial_t u) (z,t)= L[ u(\cdot,t)](z,t) + \epsilon f(u(\cdot,t),z,t)$

To approximate the long time effect of the nonlinearity by classical averaging, one can let $u=e^{Lt}w$, which leads to
\begin{equation}
	\partial_t w = \epsilon e^{-Lt} f\left( e^{Lt}w, z, t \right),
	\label{eq_classicalPreAvgPDE}
\end{equation}
and then use the approximation
\begin{equation}
	\partial_t \bar{w} = \epsilon \left\langle e^{-Lt} f\left( e^{Lt}\bar{w}, z, t \right) \right\rangle_t.
	\label{eq_classicalAvgPDE}
\end{equation}

Such an application of averaging methods to \eqref{eq_canonicalPDE} have already been extensively investigated; see for instance, \cite{mitropolsky2012asymptotic, shtaras1989averaging, verhulst1999averaging} for surveys, \cite{van1979asymptotic} for the validity of the formal extension to infinite dimensions, \cite{buitelaar1993method, krol1989galerkin} for more detailed discussions, \cite{majda1998averaging, haut2014asymptotic} for studies with applied and computational objectives, and \cite{bates1998existence, bates1999persistence} for alternative but related analyses.

The accuracy improvement idea proposed in Section \ref{sec_improvedMethod} works in this infinite dimensional setup too: let $u=e^{Lt}v-\epsilon P[v]$ instead, where $P[v]$ and $C[v]$ are functions of $z$ defined as
\begin{equation}
	P[v] := \widehat{L^{-1}} C[v],	\qquad C[v](z)=\left\langle f\left( \left(e^{Lt}v(\cdot,t)\right)(z,t), z, t\right)\right\rangle_t,
	\label{eq_improvedAvgPDE_CZandPZ}
\end{equation}
and then
\[
	\partial_t v = \epsilon e^{-L t} \left( f \left( e^{Lt} v - \epsilon P[v], z, t \right) - L P[v] \right) + \mathcal{O}(\epsilon^2).
\]
One thus uses the approximation
\begin{equation}
	\partial_t \bar{v} = \epsilon \left\langle e^{-L t} \left( f \left( e^{Lt} \bar{v} - \epsilon P[\bar{v}], z, t \right) - L P[\bar{v}] \right) \right\rangle_t
	\label{eq_improvedAvgPDE}
\end{equation}
as an improvement of \eqref{eq_classicalAvgPDE}.

Examples of \eqref{eq_canonicalPDE} include advection-reaction equations, 2nd-order wave equations with small nonlinearity, and nonlinear Schr\"odinger equations. From now on, we will be considering a specific illustrative example, which is an advection-reaction PDE in 2D space:
\begin{equation}
	\partial_t u = a \partial_x u + b \partial_y u + \epsilon f(u,x,y,t),
\end{equation}
with periodic boundary conditions $u(x,\cdot,\cdot)=u(x+L_1,\cdot,\cdot)$ and $u(\cdot,y,\cdot)=u(\cdot,y+L_2,\cdot)$. Without loss of generality, assume $a=1$ and $b=1$. For a concrete demonstration, a specific autonomous local nonlinearity was chosen: 
\[
	f(u,x,y,t) = \frac{\cos(u(x,y,t))}{1+\frac{1}{2}\cos\left(\frac{4\pi}{L_1}x\right)\sin\left(\frac{2\pi}{L_2}y\right)},
\]
for which we may already be unable to find an analytical expression of the classically averaged system \eqref{eq_classicalAvgPDE}; see Section \ref{sec_PDE_periodicCase}.

In this problem, the differential operator $L=a\partial_x+b\partial_y$, and
\begin{equation}
	(e^{Lt} v)(x,y)=v(x+at,y+bt).
	\label{eq_advectionOperator}
\end{equation}
Based on the periodic boundary condition and an assumption on the solution regularity, $v$ can be expressed in Fourier series as
\[
	v(x,y)=\sum_{j=-\infty}^\infty \sum_{k=-\infty}^\infty \hat{v}_{jk} \exp\left( i 2\pi \left(j \frac{x}{L_1}+ k\frac{y}{L_2}\right)\right).
\]
This way, it is easy to see that $L$ eigenvalues are $i 2\pi \left(\frac{j}{L_1}+\frac{k}{L_2}\right)$ for integers $j,k$, and an alternative expression to \eqref{eq_advectionOperator} is
\begin{equation}
	e^{Lt}v=\sum_{j=-\infty}^\infty \sum_{k=-\infty}^\infty \exp(i\omega_{jk}t) \hat{v}_{jk} \exp\left( i 2\pi \left(j \frac{x}{L_1}+ k\frac{y}{L_2}\right)\right),
	\label{eq_advectionOperatorFourier}
\end{equation}
where the intrinsic oscillation frequencies are $\omega_{jk}=2\pi \left(\frac{j}{L_1}+\frac{k}{L_2}\right)$.

\subsection{The periodic case}
\label{sec_PDE_periodicCase}
When $L_1/L_2$ is a rational number, denote the ratio by $\alpha/\beta$ for relatively prime integers $\alpha$ and $\beta$. It is easy to see from \eqref{eq_advectionOperatorFourier} that $e^{Lt}$ (and $e^{-Lt}$ too) is periodic with the smallest period being $T=\beta L_1=\alpha L_2=\text{lcm}(L_1,L_2)$. This periodic case will also be occasionally referred to as the resonant case.

In order to perform classical averaging, the right hand side of \eqref{eq_classicalAvgPDE} needs to be computed. Note
\begin{equation}
	e^{-Lt}f\left( (e^{Lt} w)(x,y,t), x, y \right) = e^{-Lt}f(w(x+t,y+t,t),x,y) = f(w(x,y,t),x-t,y-t).
	\label{eq_PDE_classicalAvgOriginalVectorField}
\end{equation}
Therefore, the classically averaged equation is
\begin{equation}
	\partial_t \bar{w} = \epsilon \frac{\cos \bar{w}}{\text{lcm}(L_1,L_2)} \int_0^{\text{lcm}(L_1,L_2)} \frac{1}{1+\frac{1}{2}\cos\left(\frac{4\pi}{L_1}(x-t) \right)\sin\left(\frac{2\pi}{L_2}(y-t)\right)} dt.
\end{equation}
In this resonant case, we have not found a general closed-form expression for this integral, except when $L_1=2L_2$. When $L_1=2L_2$, let $\tau=(t-x)/L_2$, then
\begin{align}
	&\quad \frac{1}{\text{lcm}(L_1,L_2)} \int_0^{\text{lcm}(L_1,L_2)} \frac{1}{1+\frac{1}{2}\cos\left(\frac{4\pi}{L_1}(x-t) \right)\sin\left(\frac{2\pi}{L_2}(y-t)\right)} dt \nonumber\\
	&= \frac{1}{2} \int_0^2 \frac{1}{1+\frac{1}{2}\cos\left(-2\pi\tau\right)\sin\left(\frac{2\pi}{L_2}(y-x)-2\pi\tau \right)} d\tau \nonumber\\
	&= \frac{1}{2} \int_0^2 \frac{1}{1+\frac{1}{4}\sin\left(\frac{2\pi}{L_2}(y-x)\right)-\frac{1}{4}\sin(4\pi\tau)} d\tau \label{eq_PDEperiodic_analyticalClassicalAvg}  \\
	&= \frac{4}{\sqrt{16\left(1+\frac{1}{4}\sin\left(\frac{2\pi}{L_2}(y-x)\right)\right)^2-1}}. \nonumber
\end{align}
When $L_1 \neq 2 L_2$, however, the denominator in \eqref{eq_PDEperiodic_analyticalClassicalAvg} contains two $\tau$-dependent terms instead of one, which make us unable to evaluate the integral. Therefore, we instead compute the right hand side of \eqref{eq_classicalAvgPDE} numerically for any given $\bar{w}$. More precisely, the exponential maps are approximated by spectral integrations, and the time averaging is numerically approximated by composite trapezoidal rule (eq. \ref{eq_numAveragingPeriodic}). These allow the classically averaged system to be constructed and then numerically integrated.

In order to perform the improved averaging, the function $C[v]$ \eqref{eq_improvedAvgPDE_CZandPZ} first needs to be computed. Unfortunately, since
\[
	f\left( (e^{Lt} v)(x,y,t), x, y \right) = f(v(x+t,y+t,t),x,y)
\]
involves nonlocality through the unknown $v$, its average $\int_0^T f(v(x+t,y+t,t),x,y) dt/T$ does not admit a closed-form expression. However, given $v$, this integral can be numerically computed, and the proposed improved approach is thus implementable via numerical averaging, similar to the classical averaging in the $L_1 \neq 2 L_2$ case. Two averages have to be computed per vector field evaluation, first for $C[v]$, and then for the right hand side of \eqref{eq_improvedAvgPDE} (where $P[v]$ is computed via a spectral discretization of $L$).

\begin{figure}[h]
\centering
\footnotesize
\subfigure[$L_1=2\sqrt{3}$, $L_2=\sqrt{3}$, in which case the classically averaged vector field is analytically available.]{
\includegraphics[width=0.48\textwidth]{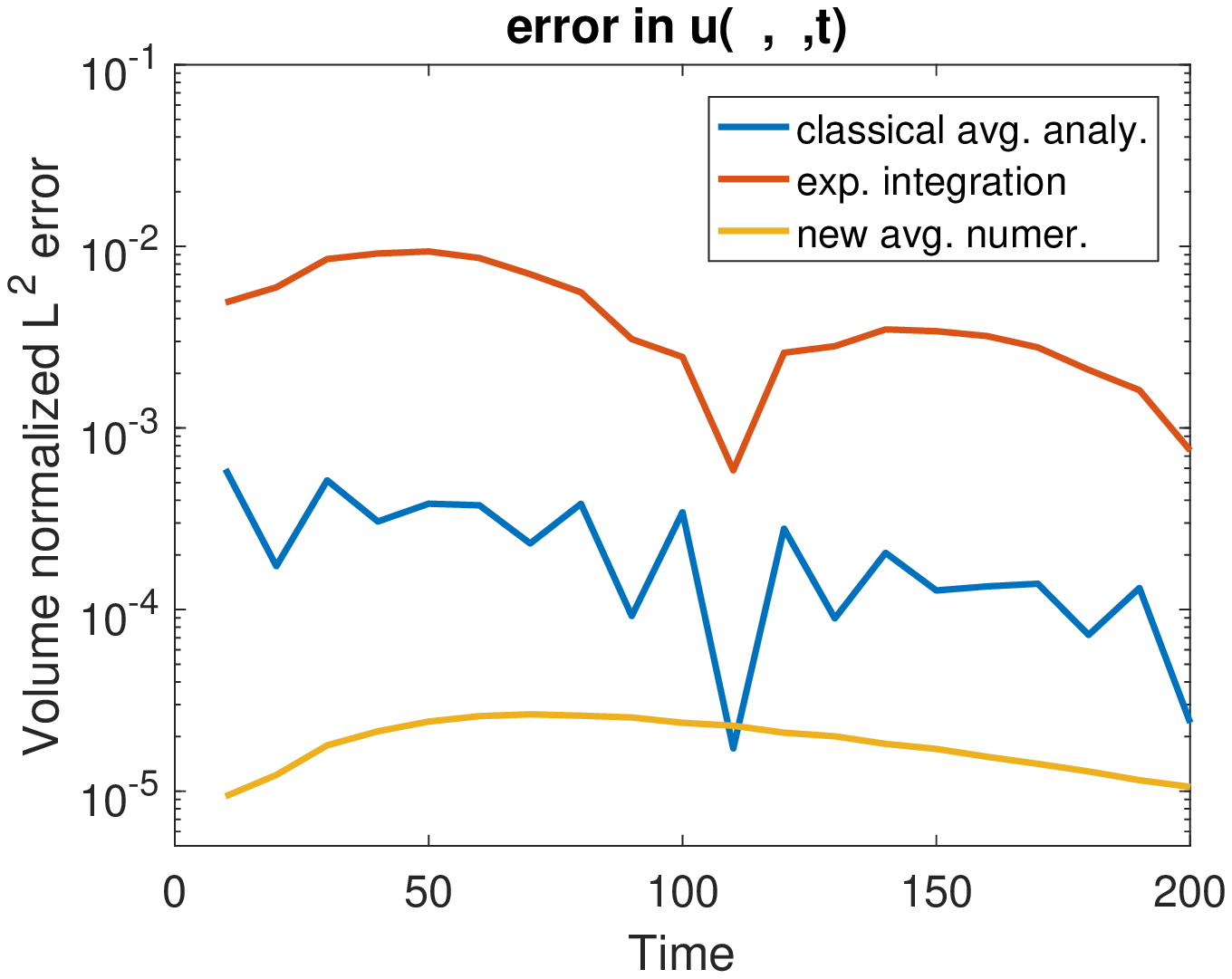}
\label{fig_PDE_resonant_L1eq2L2}
}
\subfigure[$L_1=\sqrt{2}$, $L_2=2\sqrt{2}$, in which case the classically averaged vector field is numerically computed.]{
\includegraphics[width=0.48\textwidth]{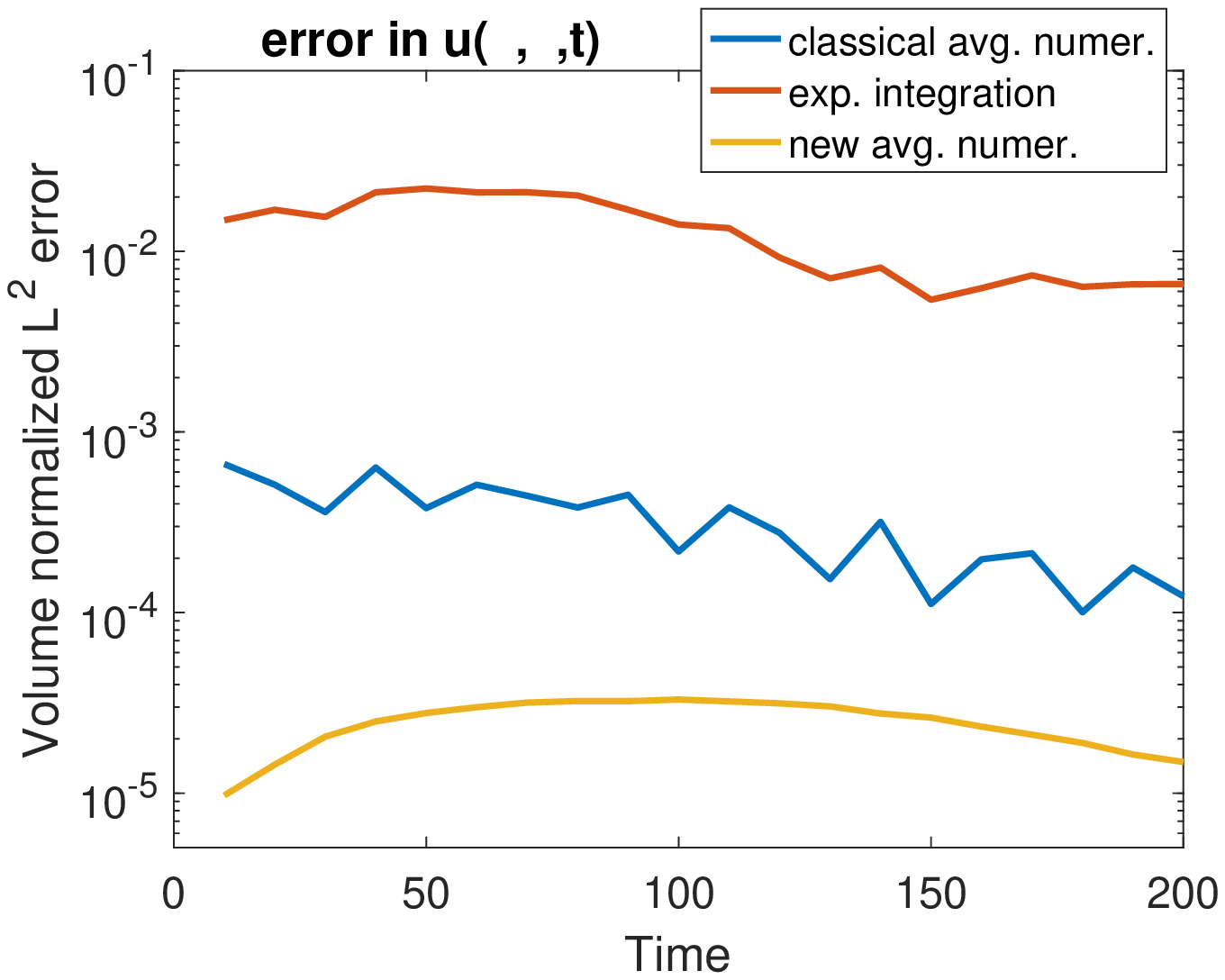}
\label{fig_PDE_resonant_L1neq2L2}
}
\caption{\footnotesize The periodic advection plus nonlinearity problem: absolute errors of classical averaging, exponential integration (with same large time step), and improved averaging}
\label{fig_PDE_resonant}
\end{figure}

Figure \ref{fig_PDE_resonant} illustrates the accuracy of the improved averaging \eqref{eq_improvedAvgPDE}. It reduces the classical averaging error by one order of magnitude, while costing $\sim$1.6x computational time (when the classically averaged system is numerically constructed).

More specifically, all numerical time-averagings in this section (the periodic case) use $N=10$ samples. The benchmark solution was obtained by RK4 time-integration and pseudospectral space-discretization, with a timestep of $0.0001$, 50 Fourier modes for $x$ and 20 for $y$. All initial conditions are $u(x,y,0)=\sin(\sin(2\pi x/L_1)+2\pi y/L_2)/4$. $\epsilon=0.01$, total simulation time is $2\epsilon^{-1}$, and all tested methods use a large timestep of $10$. Simulations based on classical and improved averaging employ RK4 to integrate the numerically averaged vector fields. A 2nd-order exponential integrator was used as a comparison, based on Strang splitting and RK2 for the nonlinear part (note for this problem an efficient 4th-order exponential integrator is less trivial to construct than that for the FPU problem, because the exact flow of $\partial_t \hat{u}=\epsilon f(\hat{u},x,y,t)$ is no longer available).

The benchmark, the exponential integrator, the classical averaging, and the improved averaging respectively spent 161.58, 0.02, 0.52, 0.82 seconds on the computation (using an Intel i7-4600 laptop with Windows 7 x64 and MATLAB R2016b). In the $L_1=2L_2$ case where the classically averaged vector field admits closed-form expression (Figure \ref{fig_PDE_resonant_L1eq2L2}), the method spent 0.03 seconds instead of 0.52 seconds; in this case, the errors of classical and improved averaging approaches, averaged over all time steps, are respectively $27.03\times 10^{-5}$ and $1.92\times 10^{-5}$, and the gain in accuracy is $\approx 14.2$-fold. In the case of Figure \ref{fig_PDE_resonant_L1neq2L2} ($2L_1=L_2$, no closed-form averaged vector field), the classical and improved errors are respectively $36.06\times 10^{-5}$ and $2.53\times 10^{-5}$, and the gain in accuracy is $\approx 14.3$-fold.

\subsection{The quasiperiodic case}
\label{sec_PDE_quasiperiodicCase}

When $L_1/L_2$ is irrational, $e^{Lt}$ and $e^{-Lt}$ are only quasiperiodic (see \eqref{eq_advectionOperatorFourier} and Section \ref{sec_PDE_periodicCase} for elaborations), and we are in a non-resonant situation.

To perform classical averaging, the time average of $f(w,x-t,y-t)$ \eqref{eq_PDE_classicalAvgOriginalVectorField} is again needed. Recall $f(w,x,y)$ is $L_1$-periodic in $x$ and $L_2$-periodic in $y$, and therefore the classically averaged equation,
\[
	\partial_t \bar{w}=\epsilon \cos\bar{w} \lim_{T\rightarrow \infty}\frac{1}{T}\int_0^T \frac{1}{1+\frac{1}{2}\cos\left(\frac{4\pi}{L_1}(x-t) \right)\sin\left(\frac{2\pi}{L_2}(y-t)\right)} dt,
\]
can be computed using Birkhoff ergodic theorem \cite{birkhoff1931proof}: the dynamics of $t \mapsto \left[\frac{x-t}{L_1},\frac{y-t}{L_2}\right]$ is ergodic on the torus with respect to uniform measure, and thus
\[
	\partial_t \bar{w}=\epsilon \cos\bar{w} \int_0^1\int_0^1 \frac{1}{1+\frac{1}{2}\cos\left(4\pi\lambda_1 \right)\sin\left(2\pi\lambda_2\right)} d\lambda_1 d\lambda_2 = \epsilon \cos\bar{w} \frac{4}{\sqrt{3}\pi} K(-1/3),
\]
where $K(\cdot)$ is the complete elliptic integral of the first kind, and $K(-1/3)\approx 1.4599$.

To perform the improved averaging \eqref{eq_improvedAvgPDE}, numerical averaging is employed for the same reason as discussed in Section \ref{sec_PDE_periodicCase}. The only difference is, as the time averaging operator is no longer over a period but $\lim_{T\rightarrow\infty}\frac{1}{T}\int_0^T$, we use the method of weighted Birkhoff averaging (\cite{das2017quantitative}, see also eq.\ref{eq_numAveragingQuasiperiodic} and \ref{eq_weightsBirkhoff}) instead of composite trapezoidal rule to approximate the time averages.

\begin{figure}[h]
\centering
\footnotesize
\subfigure[$N=100$]{
\includegraphics[width=0.48\textwidth]{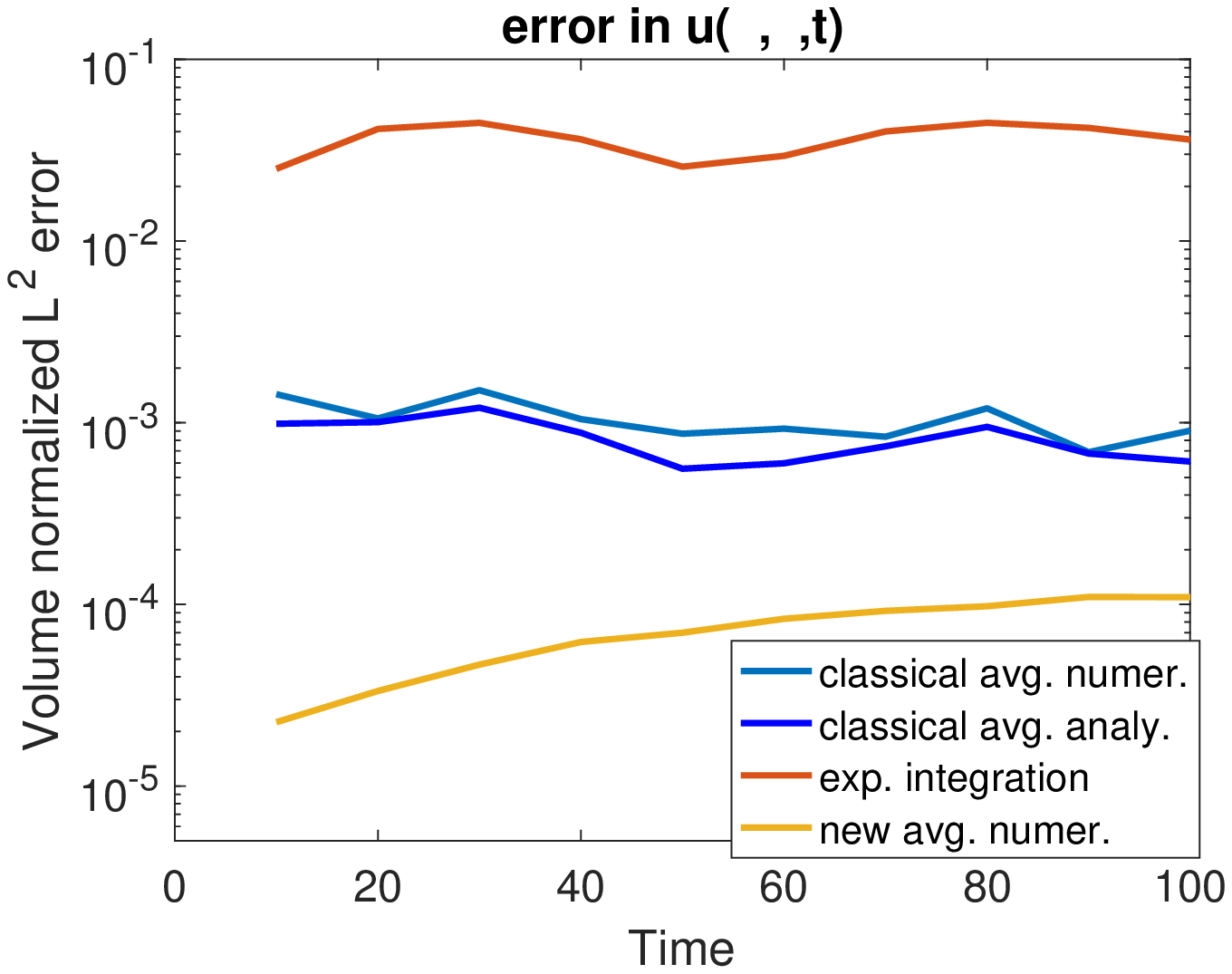}
\label{fig_PDE_nonresonant_N100}
}
\subfigure[$N=1000$]{
\includegraphics[width=0.48\textwidth]{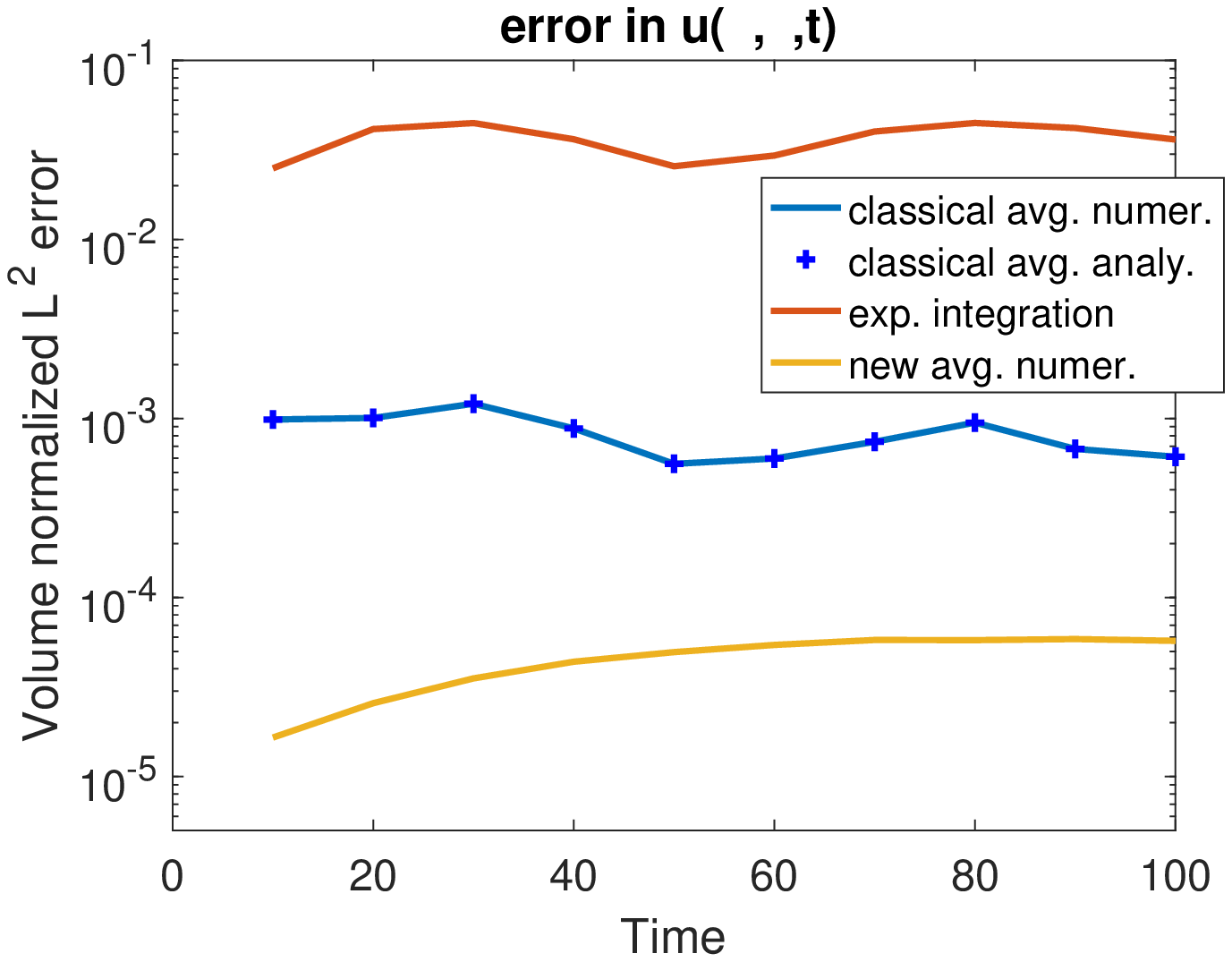}
\label{fig_PDE_nonresonant_N1000}
}
\caption{\footnotesize The quasiperiodic advection plus nonlinearity problem: absolute errors of classical averaging, exponential integration (with same large time step), and improved averaging. $L_1=2\pi$ and $L_2=1$.}
\label{fig_PDE_nonresonant}
\end{figure}

Figure \ref{fig_PDE_nonresonant} illustrates the accuracy of the improved averaging in a quasiperiodic setup. The new method again reduces the error of the classical averaging by one order of magnitude. In terms of computational cost, it takes $\sim$1.6x the computational time of numerically performed classical averaging, and $\sim$76.3x the time of analytically performed classical averaging (when the new method uses $N=100$). While it is significantly slower when compared to the analytically averaged problem, the computation is still up-scaled because its total time remains bounded and independent of $\epsilon$ for simulation till $\mathcal{O}(\epsilon^{-1})$ physical time.

In greater details, since functions to be time-averaged in this section are only quasiperiodic, more samples are used in each numerical averaging. Figure \ref{fig_PDE_nonresonant_N100} uses $N=100$ and Figure \ref{fig_PDE_nonresonant_N1000} uses $N=1000$, and in both cases a fixed $\delta t=0.17321\approx\sqrt{3}/10$ is used (see eq.\ref{eq_numAveragingQuasiperiodic}) to avoid under-sampling due to resonance with intrinsic frequencies. Other setups are as in Section \ref{sec_PDE_periodicCase}.

When $N=100$, the benchmark, the exponential integrator, the integration of the analytically obtained classical averaged system, the integration of the numerically approximated classical averaged system, and the integration of the improved averaged system respectively spent 71.78, 0.06, 0.06, 11.72, 17.78 seconds on the computation. In this case, one sees from Figure \ref{fig_PDE_nonresonant_N100} that the numerical averaging introduces approximation error in addition to the error of the analytically averaged system. Increasing to $N=1000$ makes this approximation error negligible, as Figure \ref{fig_PDE_nonresonant_N1000} shows, but the price to pay is significantly increased computational time, which multiplied from 11.72 and 17.78 seconds to 119.67 and 175.94 seconds. In terms of trajectory error (averaged over all time steps), numerically-approximated classical averaged system and improved averaged system respectively yield $80.98\times 10^{-5}$ and $10.63\times 10^{-5}$ when $N=100$ ($\approx 7.62$-fold improvement), and $80.78\times 10^{-5}$ and $4.57\times 10^{-5}$ when $N=1000$ ($\approx 17.68$-fold improvement).

Considering that $e^{Lt}$ has not only two frequencies $2\pi/L_1=1$ and $2\pi/L_2=2\pi$ but all their integer combinations (see \eqref{eq_advectionOperatorFourier}), being able to accurately average quasiperiodic vector fields containing all these frequencies (with a medium/high accuracy by using $N=100$/$1000$ samples) is rather nontrivial. We attribute this to the fast (super-polynomial) convergence of weighted Birkhoff averaging \cite{das2017quantitative}.

\section{Acknowledgment}

MT has been partially supported by NSF grant DMS-1521667 and ECCS-1829821.
MT sincerely thanks Levent Degertekin for suggesting the CPUT problem and Sushruta Surappa for generous discussions on the experimental details of CPUT. MT is grateful for encouraging discussions with Evelyn Sander, Rafael de la Llave, Chongchun Zeng, Yoshitaka Saiki, and James Yorke.

\section*{Appendix}
\paragraph{Exact expressions of the classically averaged vector field of FPU.}

For the $m=3$ case, let $X=[q_1,q_2,q_3,p_1,p_2,p_3,q_4,q_5,q_6,p_4,p_5,p_6]^T$ and $Y=e^{-\Omega t}X=[x_1,x_2,x_3,y_1,y_2,y_3,x_4,x_5,x_6,y_4,y_5,y_6]^T$, then the averaged dynamics of $Y$ is given by $\dot{\bar{Y}}=\mathcal{F}(\bar{Y})$, where

\begin{align*}
	& \mathcal{F}(Y)=\Big[ y_1, y_2, y_3, \\
	& -\left(x_1 \left(3 y_4^2+\left(2 x_1^2+3 x_4^2\right) \omega ^2\right)+(x_1-x_2) \left(3 (y_4+y_5)^2+\left(2 (x_1-x_2)^2+3 (x_4+x_5)^2\right) \omega ^2\right)\right)/(2 \omega ^2), \\
	& \big((x_1-x_2) \left(3 (y_4+y_5)^2+\left(2 (x_1-x_2)^2+3 (x_4+x_5)^2\right) \omega ^2\right) \\
		& \qquad -(x_2-x_3) \left(3 (y_5+y_6)^2+\left(2 (x_2-x_3)^2+3 (x_5+x_6)^2\right) \omega ^2\right)\big)/(2 \omega ^2), \\
	& \left( (x_2-x_3) \left(3 (y_5+y_6)^2+\left(2 (x_2-x_3)^2+3 (x_5+x_6)^2\right) \omega ^2\right)-x_3 \left(3 y_6^2+\left(2 x_3^2+3 x_6^2\right) \omega ^2\right) \right)/(2 \omega ^2), \\
	& 3 \big(\left(4 (2 y_4+y_5) x_1^2-8 x_2 (y_4+y_5) x_1+\left(2 x_4^2+2 x_5 x_4+x_5^2\right) y_4+(x_4+x_5)^2 y_5+4 x_2^2 (y_4+y_5)\right) \omega ^2 \\
		& \qquad + 2 y_4^3+3 y_5 y_4^2+3 y_5^2 y_4+y_5^3 \big) / (8 \omega ^4), \\
	& 3 \big(y_4^3+3 y_5 y_4^2+3 y_5^2 y_4+(4 (x_1-x_2)^2+(x_4+x_5)^2) \omega ^2 y_4+2 y_5^3+y_6^3+3 y_5 y_6^2 \\
		& \qquad +((x_4^2+2 x_5 x_4+2 x_5^2+x_6^2+4 (x_1^2-2 x_2 x_1+2 x_2^2+x_3^2-2 x_2 x_3)+2 x_5 x_6) y_5 \\
		& \qquad +(4 (x_2-x_3)^2+(x_5+x_6)^2) y_6) \omega ^2+3 y_5^2 y_6 \big) / (8 \omega ^4), \\
	& 3 \big(y_5^3+3 y_6 y_5^2+3 y_6^2 y_5+2 y_6^3+(4 (y_5+y_6) x_2^2-8 x_3 (y_5+y_6) x_2+(x_5+x_6)^2 y_5 \\
		& \qquad +(x_5^2+2 x_6 x_5+2 x_6^2) y_6+4 x_3^2 (y_5+2 y_6)) \omega ^2 \big) / (8 \omega ^4), \\
	& 3 \big(-(2 x_4+x_5) y_4^2-2 (x_4+x_5) y_5 y_4-(x_4+x_5) y_5^2-(4 (2 x_4+x_5) x_1^2-8 x_2 (x_4+x_5) x_1 \\
		& \qquad +4 x_2^2 (x_4+x_5)+(2 x_4+x_5) (x_4^2+x_5 x_4+x_5^2)) \omega ^2 \big) / (8 \omega ^2), \\
	& -3 \big( (4 (x_4+x_5) x_1^2-8 x_2 (x_4+x_5) x_1+4 x_3^2 (x_5+x_6)-8 x_2 x_3 (x_5+x_6) \\
		& \qquad +4 x_2^2 (x_4+2 x_5+x_6)+(x_4+2 x_5+x_6) (x_4^2+(x_5-x_6) x_4+x_5^2+x_6^2+x_5 x_6)) \omega ^2 \\
		& \qquad +x_4 (y_4+y_5)^2+x_6 (y_5+y_6)^2+x_5 (y_4^2+2 y_5 y_4+2 y_5^2+y_6^2+2 y_5 y_6) \big) / (8 \omega ^2), \\
	& 3 \big(-(x_5+x_6) y_5^2-2 (x_5+x_6) y_6 y_5-(x_5+2 x_6) y_6^2-(4 (x_5+x_6) x_2^2-8 x_3 (x_5+x_6) x_2 \\
		& \qquad +(x_5+2 x_6) (4 x_3^2+x_5^2+x_6^2+x_5 x_6)) \omega ^2 \big) /(8 \omega ^2) \Big]^T.
\end{align*}

\footnotesize
\bibliographystyle{siam}
\bibliography{../molei27}

\end{document}